vv
\RequirePackage{fix-cm}
\documentclass[smallextended]{svjour3}       
\smartqed  
\usepackage{graphicx}
\usepackage{amsmath}
\usepackage{amsfonts}
\usepackage{cite}
\usepackage{mathtools}
\usepackage{enumitem} 

\numberwithin{equation}{section}
\numberwithin{remark}{section}
\numberwithin{theorem}{section}
\numberwithin{proposition}{section}
\numberwithin{lemma}{section}
\numberwithin{definition}{section}

\usepackage[colorlinks,linkcolor=blue]{hyperref}
\hypersetup{citebordercolor={1 1 1},citecolor=blue}
\hypersetup{linkbordercolor={1 1 1},linkcolor=blue,linktocpage=true}

\begin{document}

\title{A structure-preserving parametric finite element method with optimal energy stability condition for anisotropic surface diffusion 
}

\titlerunning{A SP-PFEM with optimal energy stability condition}  

\author{Yifei Li   \and
        Wenjun Ying  \and
        Yulin Zhang 
}


\institute{Y.~Li \at
              Mathematisches Institut, Universit\"{a}t T\"{u}bingen, Auf der Morgenstelle 10, 72076, T\"{u}bingen, Germany \\
              \email{yifei.li@mnf.uni-tuebingen.de}           
           \and
           W.~Ying \at
              School of Mathematical Sciences and Institute of Natural Sciences, Shanghai Jiao Tong University, Shanghai 200240, P. R. China  \\
              \email{wying@sjtu.edu.cn}   
           \and
           Y.~Zhang \at
              School of Mathematical Sciences and Institute of Natural Sciences, Shanghai Jiao Tong University, Shanghai 200240, P. R. China\\
              \email{yulin.zhang@sjtu.edu.cn}
}

\date{Received: date / Accepted: date}

\maketitle

\begin{abstract}
 We propose and analyze a structure-preserving parametric finite element method (SP-PFEM) for the evolution of closed curves under anisotropic surface diffusion with surface energy density $\hat{\gamma}(\theta)$. Our primary theoretical contribution establishes that the condition $3\hat{\gamma}(\theta)-\hat{\gamma}(\theta-\pi)\geq 0$ is both necessary and sufficient for unconditional energy stability within the framework of local energy estimates. The proposed method introduces a symmetric surface energy matrix $\hat{\boldsymbol{Z}}_k(\theta)$ with a stabilizing function $k(\theta)$, leading to a conservative weak formulation. Its fully discretization via SP-PFEM rigorously preserves the two geometric structures: enclosed area conservation and energy dissipation unconditionally under our energy stability condition. Numerical results are reported to demonstrate the efficiency and accuracy of the proposed method, along with its area conservation and energy dissipation properties.
\end{abstract}

\section{Introduction}\label{sec:intro}

\paragraph{Background}
\textbf{Anisotropic surface diffusion} is a fundamental kinetic process in materials science, characterized by the spatially anisotropic mass transport of atoms, molecules and atomic clusters along solid material surfaces, where the directional dependence is governed by the underlying lattice structure \cite{oura2013surface}. This phenomenon has garnered increasing attention in various fields of surface/materials science, such as heterogeneous catalysis \cite{randolph2007controlling}, epitaxial growth of thin films \cite{gurtin2002interface,fonseca2014shapes}, and crystal growth of nanomaterials \cite{gilmer1972simulation,gomer1990diffusion}. Furthermore, the anisotropic surface diffusion finds numerous applications in fields such as computational geometry and solid-state physics, spanning areas like image processing \cite{clarenz2000anisotropic}, quantum dot manufacturing \cite{fonseca2014shapes} and solid-state dewetting \cite{jiang2016solid,jiang2012phase,jiang2018solid,jiang2020sharp,srolovitz1986capillary,thompson2012solid,wang2015sharp,ye2010mechanisms}.

Let $\Gamma\coloneqq\Gamma(t)\subset\mathbb{R}^2$ be an evolving closed two-dimensional (2D) curve parameterized by $\boldsymbol{X}=\boldsymbol{X}(s,t)\coloneqq(x(s,t),y(s,t))^T$, where $s$ represents the arc-length parameter and $t$ denotes time. We denote the unit outward normal vector by $\boldsymbol{n}=\boldsymbol{n}(\theta)\coloneqq(-\sin\theta,\cos\theta)^T$ and the corresponding unit tangent vector by $\boldsymbol{\tau}=\boldsymbol{\tau}(\theta)\coloneqq(\cos\theta,\sin\theta)^T$, where $\theta\in 2\pi\mathbb{T}\coloneqq\mathbb{R}/2\pi\mathbb{Z}$ represents the angle between the vertical axis and $\boldsymbol{n}$, see Fig.~\ref{fig:illust of surf diffusion}. To characterize the direction-dependent effect, an anisotropic surface energy density $\hat{\gamma}(\theta)>0$ is introduced. Subsequently, the total free energy of $\Gamma$ is defined as:
\begin{equation}\label{eqn:tot free energy}
    W_c(\Gamma)\coloneqq \int_{\Gamma}\hat{\gamma}(\theta)\,\mathrm{d}s.
\end{equation}
\begin{figure}
    \centering
    \includegraphics[width=0.618\linewidth]{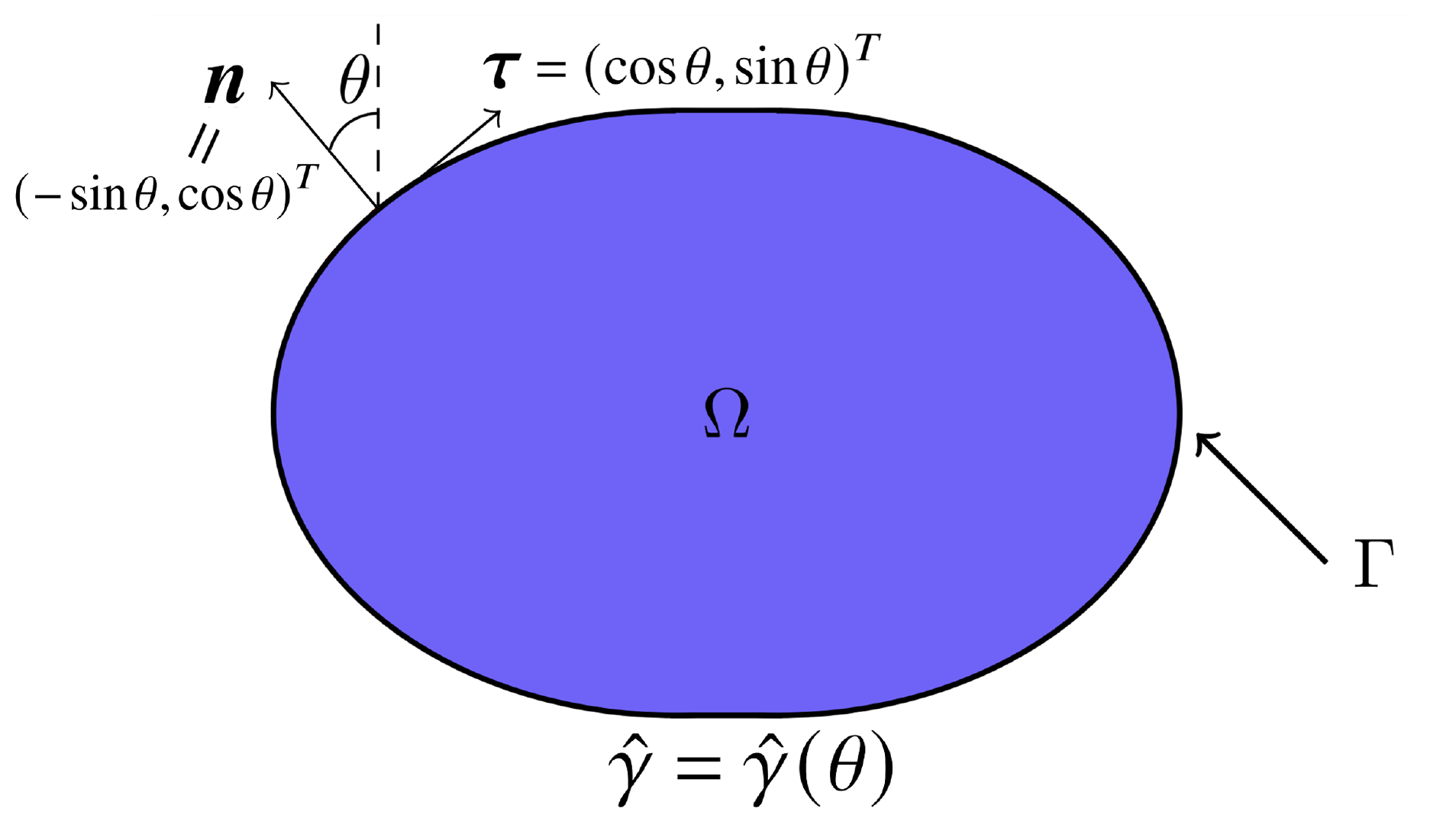}
    \caption{An illustration of a closed curve under anisotropic surface diffusion with surface energy $\hat{\gamma}(\theta)$, while $\theta$ is the angle between the $y$-axis and the unit outward normal vector $\boldsymbol{n}=\boldsymbol{n}(\theta)\coloneqq(-\sin\theta,\cos\theta)^T$. $\boldsymbol{\tau}=\boldsymbol{\tau}(\theta)\coloneqq(\cos\theta,\sin\theta)^T$ represents the unit tangent vector.}
    \label{fig:illust of surf diffusion}
\end{figure}

Following \cite{taylor1992ii,bao2024unified}, the evolution of $\Gamma$ under anisotropic surface diffusion is derived as the $H^{-1}$-gradient flow of $W_c$ \eqref{eqn:tot free energy}, which is formulated as:
\begin{equation}\label{eqn:surf diffusion}
    \partial_t\boldsymbol{X}=(\partial_{ss}\mu)\boldsymbol{n},
\end{equation}
where $\mu$ is the weighted curvature. More precisely, $\mu$ is defined by the functional derivative of $W_c(\Gamma)$ with respect to $\Gamma$ as
\begin{equation*}
    \mu\coloneqq\frac{\delta W_c(\Gamma)}{\delta\Gamma}=\lim_{\varepsilon\to 0}\frac{W_c(\Gamma^\varepsilon)-W_c(\Gamma)}{\varepsilon},
\end{equation*}
with $\Gamma^\varepsilon$ representing a small perturbation of $\Gamma$. Based on the $\hat{\gamma}(\theta)$ formulation, the weighted curvature $\mu$ also admits an explicit representation \cite{jiang2020sharp} in terms of the anisotropic surface energy density $\hat{\gamma}(\theta)$ as: 
\begin{equation}\label{eqn:weighted curvature}
    \mu=\left[\hat{\gamma}(\theta)+\hat{\gamma}^{\prime\prime}(\theta)\right]\kappa,
\end{equation}
here $\kappa\coloneqq -(\partial_{ss}\boldsymbol{X})\cdot\boldsymbol{n}$ denoting the curvature. In case of no directional dependence, i.e., $\hat{\gamma}(\theta)\equiv 1$, the weighted curvature $\mu$ reduces to $\kappa$, and equation \eqref{eqn:surf diffusion} goes to the (isotropic) surface diffusion.

The evolution equation \eqref{eqn:surf diffusion} of the anisotropic surface diffusion, being a fourth-order geometric flow, has two fundamental geometric properties: (i) the conservation of the enclosed area $A_c(t)$ by $\Gamma(t)$, and (ii) the dissipation of the total free energy $W_c(t)$. It is desirable to develop a numerical method that can preserve these geometric properties.

Different methods have been conducted on the numerical approximations of isotropic/anisotropic curvature flows over the past several decades. For example, the marker particle method \cite{du2010tangent, mayer2001numerical,wong2000periodic}, the discontinuous Galerkin method \cite{xu2009local}, the $\theta-L$ formulation method \cite{huang2021theta}, the phase-field method \cite{du2020phase,garcke2023diffuse,jiang2012phase,tang2020efficient}, the evolving surface element methods \cite{kovacs2021convergent,gui2022convergence,dziuk2007finite,duan2024new} and the parametric finite element method (PFEM) \cite{bansch2005finite,barrett2007parametric,barrett2008parametric,barrett2020parametricbook,li2021energy,bao2021structure,bao2017parametric,jiang2016solid,barrett2008variational}. Among these approaches, the PFEM presents significant theoretical advantages from the structure-preserving perspective. The energy-stable PFEM (ES-PFEM) proposed by Barret, Garcke, and N{\"u}rnberg \cite{barrett2007parametric,barrett2008variational,barrett2008parametric}, also termed the BGN method, established the first rigorous proof for preserving energy stability in isotropic surface diffusion. Subsequently, Bao and Zhao built upon this work to propose a structure-preserving PFEM (SP-PFEM) that simultaneously maintains energy stability and area conservation at the fully discrete level \cite{bao2021structure,bao2023structure,bao2022volume}.

The extension of these structure-preserving PFEMs to anisotropic surface energies originated with a series of works by Barret, Garcke, and N{\"u}rnberg \cite{barrett2008variational,barrett2008numerical}, where they successfully adapted the energy stability to cases with a specific Riemannian-like metric form of surface energy. In \cite{li2021energy}, Bao and Li constructed a surface energy matrix $G(\theta)$, which extended the ES-PFEM from the specific forms to a much broader class of anisotropic surface energies. However, the energy stability conditions are relatively complex and restrictive.

A significant advancement in the extension of the SP-PFEM to anisotropic surface energies is adding a stabilizing functions $k(\theta)$ in the surface energy matrix. Building upon this advancement, Bao and Li \cite{bao2024structure} established an analytical framework, demonstrating that energy stability follows from the satisfaction of a local energy estimates, where the stabilizing function $k(\theta)$ is required to greater than a bounded minimal stabilizing function $k_0(\theta)$. In fact, the energy stability proofs in all existing structure-preserving/energy-stable PFEMs \cite{bao2024unified,zhang2024stabilized,li2021energy,bao2024curved,li2024structure} can be recast within this framework. Consequently, the different energy stability conditions on the surface energy $\hat{\gamma}$ emerge from the choice of surface energy matrix and the analytical techniques employed in establishing these local energy estimates. For instance, in the symmetrized SP-PFEM \cite{bao2023symmetrized2D,bao2023symmetrized3D}, the author introduced a symmetrized surface energy $\boldsymbol{Z}_k$, and proved the stability condition $\hat{\gamma}(\theta)=\hat{\gamma}(\theta-\pi)$  through an application of Cauchy's inequality. The minimal stabilizing functions $k_0(\theta)$ are estimated for several $\hat{\gamma}(\theta)$. Subsequently, by introducing another surface energy matrix $\boldsymbol{G}_k$, the works in \cite{bao2024structure,bao2024unified,bao2024curved,li2024structure} achieved an improved stability condition $3\hat{\gamma}(\theta) - \hat{\gamma}(\theta-\pi) > 0$ via refined analytical techniques. Very recently, the energy stability condition is improved to $3\hat{\gamma}(\theta) - \hat{\gamma}(\theta-\pi) \geq 0$ and $\hat{\gamma}'(\theta^*)=0$ when $3\hat{\gamma}(\theta^*) - \hat{\gamma}(\theta^*-\pi) = 0$ \cite{zhang2024stabilized}, yet the explicit characterization of $k_0(\theta)$ remained unknown. 

On the other hand, the analysis in \cite{bao2024unified} initially established for the surface energy matrix $\boldsymbol{G}_k$ that
\begin{equation}\label{eqn:energy stab cond b}
3\hat{\gamma}(\theta)-\hat{\gamma}(\theta-\pi)\geq 0,\qquad\forall\theta\in 2\pi\mathbb{T}
\end{equation}
serves as a necessary condition for the local energy estimate. Inspired by their proof, our Remark~\ref{rem:indep of Z} demonstrates that this condition is further independent of the specific construction of surface energy matrices $\boldsymbol{G}_k$ or $\boldsymbol{Z}_k$. Therefore, the energy stability condition in \cite{zhang2024stabilized} is almost optimal except for the extra condition $\hat{\gamma}'(\theta^*)=0$. This naturally raises three fundamental questions: (1) whether this necessary condition is also sufficient for energy stability, (2) if so, which surface energy matrix, coupled with appropriate analytical techniques, would achieve this optimal energy stability condition, and (3) how to explicitly characterize the minimal stabilizing function $k_0(\theta)$.

\paragraph{Main results} In this paper, we propose and analyze a structure-preserving PFEM for simulating anisotropic surface diffusion in two dimensions. The proposed SP-PFEM preserves the area conservation and achieves unconditional energy stability under the energy stability condition \eqref{eqn:energy stab cond b}, without any additional constraints necessary. This establishes the sufficiency of the necessary condition \eqref{eqn:energy stab cond b}, thereby showing its optimality within the framework of local energy estimates.

\paragraph{Structure of the paper} In Section~\ref{sec:conservative form}, we propose a conservative form and derive a weak formulation for anisotropic surface diffusion by introducing a symmetric surface energy matrix $\hat{\boldsymbol{Z}}_k(\theta)$. In Section~\ref{sec:SP-PFEM}, a full discretization by a PFEM is presented for the weak formulation. Meanwhile, we state the structure-preserving property of the method. In Section~\ref{sec:loc energy est}, a minimal stabilizing function $k_0(\theta)$ is defined. Assuming its existence, we establish the local energy estimate, thereby proving the energy stability of the proposed PFEM. Section~\ref{sec:existence of k0} provides a detailed proof for existence of $k_0(\theta)$. Furthermore, we formulate a sharp estimate lemma and derive a global upper bound for $k_0(\theta)$. Section~\ref{sec:numer} contains numerous numerical results to validate the accuracy, efficiency, and the structure-preserving property of the proposed PFEM. Finally, we summarize some conclusions in Section~\ref{sec:conclusion}.

\section{Conservative form and weak formulation}\label{sec:conservative form}

\subsection{Conservative form}\label{subsec:conservative form}

In order to derive a weak formulation for the anisotropic surface diffusion \eqref{eqn:surf diffusion}--\eqref{eqn:weighted curvature}, the following surface energy matrix is introduced: 

\begin{definition}[symmetric surface energy matrix]
    The symmetric surface energy matrix $\hat{\boldsymbol{Z}}_k(\theta)$ is given as \begin{equation}\label{def:surf energy mat}
     \begin{aligned}
        \hat{\boldsymbol{Z}}_k(\theta)&\coloneqq\begin{pmatrix}
        \hat{\gamma}(\theta)-\hat{\gamma}^\prime(\theta)\sin 2\theta & \hat{\gamma}^\prime(\theta)\cos 2\theta\\[0.5em]
        \hat{\gamma}^\prime(\theta)\cos 2\theta & \hat{\gamma}(\theta)+\hat{\gamma}^\prime(\theta)\sin 2\theta
    \end{pmatrix}\\
    &\quad+k(\theta)\begin{pmatrix}
        \sin^2\theta & -\cos\theta\sin\theta\\[0.5em]
        -\cos\theta\sin\theta & \cos^2\theta
    \end{pmatrix},\qquad\forall\theta\in 2\pi\mathbb{T}.
    \end{aligned}
    \end{equation}
\end{definition} Here, $k:2\pi\mathbb{T}\to\mathbb{R}_{\geq 0}$ is a non-negative stabilizing function that can be prespecified.

\begin{theorem}
    With the surface energy matrix \eqref{def:surf energy mat}, the following geometric identity holds: \begin{equation}\label{eqn:geo id}
        \mu\boldsymbol{n}+\partial_s\Bigl(\hat{\boldsymbol{Z}}_k(\theta)\partial_s\boldsymbol{X}\Bigr)=\boldsymbol{0}.
    \end{equation}
\end{theorem}

\begin{proof}
    First, recall Fig.~\ref{fig:illust of surf diffusion} that $\boldsymbol{n}=(-\sin\theta,\cos\theta)^T$ and $\boldsymbol{\tau}\coloneqq\partial_s\boldsymbol{X}=(\cos\theta,\sin\theta)^T$.

    Similar to derivations in \cite[Theorem 2.1]{zhang2024stabilized} or \cite[(2.10)]{li2021energy}, we have \begin{equation}\label{eqn:weighted curvature vec}
        \mu\boldsymbol{n}=-\partial_s\Bigl(\hat{\gamma}(\theta)\partial_s\boldsymbol{X}+\hat{\gamma}^\prime(\theta)\boldsymbol{n}\Bigr).
    \end{equation} Note that $\boldsymbol{n}^T\partial_s\boldsymbol{X}=\boldsymbol{n}\cdot\boldsymbol{\tau}\equiv 0$, thus $\partial_s\left(k(\theta)\boldsymbol{n}\boldsymbol{n}^T\partial_s\boldsymbol{X}\right)$ vanishes. Thus, \eqref{eqn:weighted curvature vec} can be reformulated as \begin{equation}\label{eqn:mu*n}
        \mu\boldsymbol{n}=-\partial_s\Bigl(\hat{\gamma}(\theta)\partial_s\boldsymbol{X}+\hat{\gamma}^\prime(\theta)\boldsymbol{n}+k(\theta)\boldsymbol{n}\boldsymbol{n}^T\partial_s\boldsymbol{X}\Bigr).
        \end{equation}
    
    On the other hand, denote \begin{equation}
        L_\theta\coloneqq\begin{pmatrix}
            -\sin2\theta & \cos2\theta\\[0.5em]
            \cos2\theta & \sin 2\theta
        \end{pmatrix},
    \end{equation} and acts it on the tangent vector $\partial_s\boldsymbol{X}=\boldsymbol{\tau}$,  \begin{equation}
        \begin{aligned}
            L_\theta\partial_s\boldsymbol{X}&=\begin{pmatrix}
                -\sin 2\theta & \cos 2\theta\\[0.5em]
                \cos 2\theta & \sin 2\theta
            \end{pmatrix}\begin{pmatrix}
                \cos\theta\\[0.5em]
                \sin\theta
            \end{pmatrix}=\begin{pmatrix}
                -\sin\theta\\[0.5em]
                \cos\theta
            \end{pmatrix}=\boldsymbol{n}.
        \end{aligned}
    \end{equation}
    Therefore, 
    \begin{equation}\label{eqn:Gk*tau}
        \begin{aligned}
                \hat{\boldsymbol{Z}}_k(\theta)\partial_s\boldsymbol{X}&=\Bigl(\hat{\gamma}(\theta)I_2+\hat{\gamma}^\prime(\theta)L_\theta+k(\theta)\boldsymbol{n}\boldsymbol{n}^T\Bigr)\partial_s\boldsymbol{X}\\
                &=\hat{\gamma}(\theta)\partial_s\boldsymbol{X}+\hat{\gamma}^\prime(\theta)\boldsymbol{n}+k(\theta)\boldsymbol{n}\boldsymbol{n}^T\partial_s\boldsymbol{X},
            \end{aligned}
        \end{equation} which leads to the desired equation \eqref{eqn:geo id} by substituting \eqref{eqn:Gk*tau} into \eqref{eqn:mu*n}.\qed
\end{proof}

With the help of the geometric identity \eqref{eqn:geo id}, the anisotropic surface diffusion equation \eqref{eqn:surf diffusion}--\eqref{eqn:weighted curvature} can be rewritten into a conservative form: \begin{subequations}\label{eqn:conservative form}
    \begin{align}
        &\partial_t\boldsymbol{X}\cdot\boldsymbol{n}-\partial_{ss}\mu=0,\qquad 0<s<L(t),\qquad\forall t\geq 0,\label{eqn:conservative form a}\\
        &\mu\boldsymbol{n}+\partial_s\Bigl(\hat{\boldsymbol{Z}}_k(\theta)\partial_s\boldsymbol{X}\Bigr)=\boldsymbol{0},\label{eqn:conservative form b}
    \end{align}
\end{subequations} where $L(t)\coloneqq\int_{\Gamma(t)} 1\,\mathrm{d}s$ represents the length of $\Gamma(t)$.

\subsection{Weak formulation} 

Let $\mathbb{I}\coloneqq [0,1]$ be the unit interval, and the evolving curve $\Gamma(t)$ is parameterized as \begin{equation}
    \Gamma(t)\coloneqq\boldsymbol{X}(\rho,t)=(x(\rho,t),y(\rho,t))^T\colon\mathbb{I}\times\mathbb{R}^+\to\mathbb{R}^2,
\end{equation} with a time-independent variable $\rho\in\mathbb{I}$. Then the arclength parameter $s$ can be given as $s(\rho,t)=\int_0^\rho|\partial_\rho\boldsymbol{X}|\,\mathrm{d}q$ satisfying $\partial_\rho s=|\partial_\rho\boldsymbol{X}|$. (We will not discriminate $\boldsymbol{X}(\rho,t)$ and $\boldsymbol{X}(s,t)$ for representing $\Gamma(t)$ if there's no misunderstanding.)

Introducing the following functional space \begin{equation}
    L^2(\mathbb{I})\coloneqq\left\{u\colon\mathbb{I}\to\mathbb{R}\mid\int_{\Gamma(t)}|u(s)|^2\,\mathrm{d}s=\int_{\mathbb{I}}|u(s(\rho,t))|^2\partial_\rho s\,\mathrm{d}s<+\infty\right\},
\end{equation} equipped with the $L^2$-inner product \begin{equation}
    \Bigl(u,v\Bigr)_{\Gamma(t)}\coloneqq\int_{\Gamma(t)}u(s)v(s)\,\mathrm{d}s=\int_{\Gamma(t)}u(s(\rho,t))v(s(\rho,t))\partial_\rho s\,\mathrm{d}\rho,
\end{equation} for any scalar or vector valued functions. And the Sobolev spaces are defined as \begin{subequations}
    \begin{align}
        H^1(\mathbb{I})&\coloneqq\left\{u\colon\mathbb{I}\to\mathbb{R}\mid u\in L^2(\mathbb{I}),\,\,\text{and}\,\,\partial_\rho u\in L^2(\mathbb{I})\right\},\\
        H_p^1(\mathbb{I})&\coloneqq\left\{u\in H^1(\mathbb{I})\mid u(0)=u(1)\right\}.
    \end{align}
\end{subequations} 
Extensions of above definitions to the vector-valued functions in $[L^2(\mathbb{I})]^2$, $[H^1(\mathbb{I})]^2$ and $[H_p^1(\mathbb{I})]^2$ are straightforward.

Multiplying a test function $\varphi\in H_p^1(\mathbb{I})$ to \eqref{eqn:conservative form a}, then integrating over $\Gamma(t)$ and taking integration by parts, we obtain \begin{equation}\label{eqn:weak deduce 1}
    \Bigl(\boldsymbol{n}\cdot\partial_t\boldsymbol{X},\varphi\Bigr)_{\Gamma(t)}+\Bigl(\partial_s\mu,\partial_s\varphi\Bigr)_{\Gamma(t)}=0.
\end{equation} Similarly, by taking an inner product with a test function $\boldsymbol{\omega}=(\omega_1,\omega_2)^T\in[H^1_p(\mathbb{I})]^2$ to \eqref{eqn:conservative form b} and integrating by parts, we deduce \begin{equation}\label{eqn:weak deduce 2}
    \begin{aligned}
        0&=\Bigl(\mu\boldsymbol{n}+\partial_s\left(\hat{\boldsymbol{Z}}_k(\theta)\partial_s\boldsymbol{X}\right),\boldsymbol{\omega}\Bigr)_{\Gamma(t)}\nonumber\\
        &=\Bigl(\mu\boldsymbol{n},\boldsymbol{\omega}\Bigr)_{\Gamma(t)}-\Bigl(\hat{\boldsymbol{Z}}_k(\theta)\partial_s\boldsymbol{X},\partial_s\boldsymbol{\omega}\Bigr)_{\Gamma(t)}.
    \end{aligned}
\end{equation}

Combining \eqref{eqn:weak deduce 1} and \eqref{eqn:weak deduce 2}, a weak formulation for anisotropic surface diffusion \eqref{eqn:surf diffusion}--\eqref{eqn:weighted curvature} with an initial condition $\boldsymbol{X}(s,0)=\boldsymbol{X}_0(s)=(x_0(s),y_0(s))^T$ reads: given an initial closed curve $\Gamma(0)=\boldsymbol{X}(\cdot,0)=\boldsymbol{X}_0\in [H^1_p(\mathbb{I})]^2$, find the solution $\left(\boldsymbol{X}(\cdot,t),\mu(\cdot,t)\right)\in [H^1_p(\mathbb{I})]^2\times H^1_p(\mathbb{I})$ satisfying \begin{subequations}\label{eqn:weak formulation}
    \begin{align}
        &\Bigl(\boldsymbol{n}\cdot\partial_t\boldsymbol{X},\varphi\Bigr)_{\Gamma(t)}+\Bigl(\partial_s\mu,\partial_s\varphi\Bigr)_{\Gamma(t)}=0,\qquad\forall\varphi\in H^1_p(\mathbb{I}),\label{eqn:weak formulation a}\\
        &\Bigl(\mu\boldsymbol{n},\boldsymbol{\omega}\Bigr)_{\Gamma(t)}-\Bigl(\hat{\boldsymbol{Z}}_k(\theta)\partial_s\boldsymbol{X},\partial_s\boldsymbol{\omega}\Bigr)_{\Gamma(t)}=0,\qquad\forall\boldsymbol{\omega}\in [H^1_p(\mathbb{I})]^2.\label{eqn:weak formulation b}
    \end{align}
\end{subequations}

\begin{remark}
    The surface energy matrix $\hat{\boldsymbol{Z}}_k(\theta)$ can be transformed into $\boldsymbol{Z}_k(\boldsymbol{n})$ in \cite{bao2023symmetrized2D}, by the one-to-one correspondence $\boldsymbol{n}\coloneqq \boldsymbol{n}(\theta)=(-\sin\theta,\cos\theta)^T$.
\end{remark}

\subsection{Area conservation and energy dissipation}

Suppose the solution of the weak formulation \eqref{eqn:weak formulation} be $(\boldsymbol{X}(\cdot,t),\mu(\cdot,t))$, and the evolving curve $\Gamma(t)$ is given by $\Gamma(t) = \boldsymbol{X}(\cdot,t) = (x(\cdot,t),y(\cdot,t))^T$. Let $A_c(t)$ be the area of the region enclosed by $\Gamma(t)$ and $W_c(t)$ be the total free energy, respectively, which are formally defined as \begin{equation}
    A_c(t)\coloneqq\int_{\Gamma(t)}y(s,t)\partial_sx(s,t)\,\mathrm{d}s,\qquad W_c(t)=\int_{\Gamma(t)}\hat{\gamma}(\theta)\,\mathrm{d}s.
\end{equation} 

The solution of \eqref{eqn:weak formulation} satisfies the following structure-preserving properties:

\begin{proposition}[area conservation and energy dissipation]\label{prop:area conser. energy dissip.}
    The area $A_c(t)$ of the solution $\left(\boldsymbol{X}(\cdot,t),\mu(\cdot,t)\right)\in[H_p^1(\mathbb{I})]^2\times H_p^1(\mathbb{I})$ given by \eqref{eqn:weak formulation} is conserved, and the total free energy $W_c(t)$ is dissipative, i.e. \begin{equation}
        A_c(t)\equiv A_c(0),\qquad W_c(t)\leq W_c(t_1)\leq W_c(0),\qquad\forall t\geq t_1\geq 0.
    \end{equation}
\end{proposition}

The proof for Proposition~\ref{prop:area conser. energy dissip.} is similar to \cite[Proposition 3.1]{zhang2024stabilized}. Details are omitted.

\section{A structure-preserving PFEM discretization}\label{sec:SP-PFEM}

Consider a positive integer $N>2$ and let $h=1/N$ be the mesh size. We partition the unit interval as $\mathbb{I}=[0,1]\coloneqq\cup_{j=1}^NI_j$ with sub-intervals $I_j\coloneqq [\rho_{j-1},\rho_j]$ and grid points $\rho_j=jh$ for $j=1,2,\dots,N$.

Let us introduce the finite element subspaces of $H^1(\mathbb{I})$ as: \begin{subequations}
    \begin{align}
        \mathbb{K}^h&\coloneqq\left\{u^h\in C(\mathbb{I})\mid u^h|_{I_j}\in\mathcal{P}^1(I_j),\forall 1\leq j\leq N\right\}\subseteq H^1(\mathbb{I}),\\
        \mathbb{K}^h_p&\coloneqq\left\{u^h\in\mathbb{K}^h\mid u^h(0)=u^h(1)\right\},
    \end{align}
\end{subequations} where $\mathcal{P}^1(I_j)$ stands for the space of polynomials defined on $I_j$ with degree $\leq 1$. 

Let $\tau$ be the uniform time step, and the approximation of $\Gamma(t)=\boldsymbol{X}(\cdot,t)$ at the $m^{\text{th}}$ discrete time level $t_m=m\tau$ be $\Gamma^m=\boldsymbol{X}^m(\cdot)=(x^m(\cdot),y^m(\cdot))^T\in[\mathbb{K}^h_p]^2,\,\,m=0,1,2,\dots$. Suppose the polygonal curve $\Gamma^m$ is composed by ordered line segments $\{\boldsymbol{h}_j^m\}_{j=1}^N$, i.e. \begin{equation}
    \Gamma^m=\bigcup_{j=1}^N\boldsymbol{h}_j^m,\qquad\text{with}\quad \boldsymbol{h}_j^m=(h^m_{j,x},h^m_{j,y})^T\coloneqq\boldsymbol{X}^m(\rho_j)-\boldsymbol{X}^m(\rho_{j-1})
\end{equation} for $j=1,2,\dots,N$. The unit tangential vector $\boldsymbol{\tau}^m$, the outward unit normal vector $\boldsymbol{n}^m$ and the inclination angle $\theta^m$ are constant on each interval $I_j$, which can be computed as \begin{equation}
        \boldsymbol{\tau}^m|_{I_j}=\frac{\boldsymbol{h}_j^m}{|\boldsymbol{h}_j^m|}\coloneqq\boldsymbol{\tau}_{j}^m,\qquad \boldsymbol{n}^m|_{I_j}=-(\boldsymbol{\tau}^m_j)^\perp=-\frac{(\boldsymbol{h}_j^m)^\perp}{|\boldsymbol{h}_j^m|}\coloneqq\boldsymbol{n}^m_j,
\end{equation} and \begin{equation}
    \theta^m|_{I_j}\coloneqq\theta^m_j,\qquad\text{satisfying}\quad \cos\theta^m_j=\frac{h_{j,x}^m}{|\boldsymbol{h}_j^m|},\quad\sin\theta^m_j=\frac{h_{j,y}^m}{|\boldsymbol{h}_j^m|}.
\end{equation} The mass-lumped inner product $(\cdot,\cdot)^h_{\Gamma^m}$ and discretized differential operator $\partial_s$ on $\Gamma^m$ for scalar-/vector-valued functions are defined as \begin{subequations}
    \begin{align}
        &\Bigl(f,g\Bigr)^h_{\Gamma^m}\coloneqq\sum_{j=1}^N\frac{|\boldsymbol{h}_j^m|}{2}(f(\rho_j)g(\rho_j)+f(\rho_{j-1})g(\rho_{j-1})),\label{eqn:mass-lumped inner product}\\
        &\partial_s f|_{I_j}\coloneqq\frac{f(\rho_j)-f(\rho_{j-1})}{|\boldsymbol{h}_j^m|},\qquad\forall 0\leq j\leq N.\label{eqn:discretized differential operator}
    \end{align}
\end{subequations}

Following ideas in \cite{bao2023structure,jiang2021perimeter} to design a volume-preserving scheme for the surface diffusion, by adopting the explicit-implicit Euler method for time discretization, a structure-preserving PFEM for the anisotropic surface diffusion \eqref{eqn:surf diffusion} is given as: for a given initial curve $\Gamma^0=\boldsymbol{X}^0(\cdot)\in[\mathbb{K}^h_p]^2$, find the solution $\left(\boldsymbol{X}^{m+1}(\cdot),\mu^{m+1}(\cdot)\right)\in[\mathbb{K}^h_p]^2\times\mathbb{K}^h_p,\,\,m=0,1,2,\dots$ satisfying 
\begin{subequations}\label{eqn:SP-PFEM}
    \begin{align}
            &\Bigl(\boldsymbol{n}^{m+\frac{1}{2}}\cdot \frac{\boldsymbol{X}^{m+1}-\boldsymbol{X}^m}{\tau}, \varphi^h\Bigr)_{\Gamma^m}^h+\Bigl(\partial_s\mu^{m+1}, \partial_s\varphi^h\Bigr)_{\Gamma^m}^h=0,\qquad \forall\varphi^h\in \mathbb{K}_p^h,\\
            &\Bigl(\mu^{m+1}\boldsymbol{n}^{m+\frac{1}{2}},\boldsymbol{\omega}^h\Bigr)_{\Gamma^m}^h-\Bigl(\hat{\boldsymbol{Z}}_k(\theta^m)\partial_s\boldsymbol{X}^{m+1},\partial_s\boldsymbol{\omega}^h\Bigr)_{\Gamma^m}^h=0,\qquad\forall\boldsymbol{\omega}^h\in\mathbb[\mathbb{K}^h_p]^2,
    \end{align}
\end{subequations} where \begin{equation}
    \boldsymbol{n}^{m+\frac{1}{2}}=-\frac{1}{2}\left(\partial_s\boldsymbol{X}^m+\partial_s\boldsymbol{X}^{m+1}\right)^\perp=-\frac{1}{2|\partial_\rho\boldsymbol{X}^m|}\left(\partial_\rho\boldsymbol{X}^m+\partial_\rho\boldsymbol{X}^{m+1}\right)^\perp.
\end{equation}

\begin{remark}
    The fully implicit scheme is solved numerically by the Newton's method.  The choice of $\boldsymbol{n}^{m+\frac{1}{2}}$ 
    is essential for the area conservation in the discrete level. 
\end{remark}

\subsection{Structure-preserving property of the SP-PFEM}\label{subsec:SP-PFEM properties}

Denote $A_c^m$ the area of the region enclosed by polygonal curve $\Gamma^m$ and $W_c^m$ the total free energy, respectively, which are formally defined as 
\begin{subequations}
    \begin{align}
        &A_c^m\coloneqq\frac{1}{2}\sum_{j=1}^N(x_j^m-x_{j-1}^m)(y_j^m+y_{j-1}^m),\\
        &W_c^m\coloneqq\sum_{j=1}^N\hat{\gamma}(\theta_j^m)|\boldsymbol{h}_j^m|,
    \end{align}
\end{subequations}
where $x_j^m\coloneqq x^m(\rho_j),\,y_j^m\coloneqq y^m(\rho_j), \,  j=1,2,\dots, N$.

\begin{theorem}[area conservation and unconditional energy stability]\label{thm:structure-preserving}
    Suppose $\hat{\gamma}(\theta)\in C^2(2\pi\mathbb{T})$ and satisfies the optimal energy stability condition \eqref{eqn:energy stab cond b}. The SP-PFEM \eqref{eqn:SP-PFEM} is area conservative and unconditional energy dissipative with sufficiently large $k(\theta)$, i.e. \begin{equation}\label{eqn:structure-preserving}
        A_c^{m+1}=A_c^m=\cdots=A_c^0,\qquad W_c^{m+1}\leq W_c^m\leq \cdots\leq W_c^0 \qquad\forall m\geq 0.
    \end{equation}
\end{theorem}

For the proof of area conservation, we refer the reader to \cite[Theorem 2.1]{bao2021structure} by Bao and Zhao for surface diffusion. Detailed proof of energy dissipation will appear in the next section.

\begin{remark}
    With the adoption of $\gamma(\boldsymbol{n})$, SP-PFEM \eqref{eqn:SP-PFEM} can be transformed to the symmetrized SP-PFEM in \cite{bao2023symmetrized2D}. 
    It is a significant improvement compared to the original energy stability condition $\gamma(\boldsymbol{n})=\gamma(-\boldsymbol{n})$ or $\hat{\gamma}(\theta) = \hat{\gamma}(\theta-\pi)$ in \cite{bao2023symmetrized2D}. Our analysis within the $\hat{\gamma}(\theta)$ formulation indicates that, the symmetry condition $\hat{\gamma}(\theta) = \hat{\gamma}(\theta-\pi)$ of the symmetrized SP-PFEM could be improved to optimal condition \eqref{eqn:energy stab cond b}, without any extra condition.
\end{remark}

\section{Local energy estimate and the unconditional energy stability}\label{sec:loc energy est}

\subsection{The minimal stabilizing function}

 Introduce the following auxiliary functions, \begin{subequations}\label{eqn:auxiliary func}
    \begin{align}
        P_\alpha(\phi,\theta)&\coloneqq\hat{\gamma}(\theta)-\hat{\gamma}^\prime(\theta)\sin2\phi+\alpha\sin^2\phi,\label{eqn:auxiliary func P}\\
        Q(\phi,\theta)&\coloneqq\hat{\gamma}(\theta-\phi)+\hat{\gamma}(\theta)\cos\phi-\hat{\gamma}^\prime(\theta)\sin\phi. \label{eqn:auxiliary func Q}
    \end{align}
\end{subequations} 

Thus, the minimal stabilizing function is defined by adopting the auxiliary functions $P_\alpha,Q$: \begin{equation}\label{eqn:minimal stab func}
    k_0(\theta)\coloneqq\inf\left\{\alpha\geq 0\mid 4\hat{\gamma}(\theta)P_\alpha(\phi,\theta)\geq Q^2(\phi,\theta),\,\,\forall\phi\in 2\pi\mathbb{T}\right\}.
\end{equation}

The following theorem ensures the existence of $k_0(\theta)$ and provides an upper bound that offers practical guidance for applications: \begin{theorem}\label{thm:minimal stab func well-defined}
    For $\hat{\gamma}(\theta)$ satisfying \eqref{eqn:energy stab cond b}, the minimal stabilizing function $k_0(\theta)$, as given in \eqref{eqn:minimal stab func}, is well-defined. Furthermore, we have the following estimate: \begin{equation}\label{eqn:upper bound}
    k_0(\theta)\leq \frac{1}{4\hat{\gamma}(\theta)}\left[A^2(\theta)+4\hat{\gamma}(\theta)A(\theta)+4|\hat{\gamma}^{\prime}(\theta)|^2\right]<\infty,
\end{equation} where \begin{equation}\label{eqn:A theta}
    A(\theta)\coloneqq\frac{\pi^2}{8}\left(5\sup_{2\pi\mathbb{T}}|\hat{\gamma}^{\prime\prime}|+5|\hat{\gamma}^\prime(\theta)|+\hat{\gamma}(\theta)\right).
\end{equation}
\end{theorem}

Detailed proof of Theorem~\ref{thm:minimal stab func well-defined} will be given in Section~\ref{sec:existence of k0}.

\begin{remark}
    In previous studies, the existence of $k_0(\theta)$ was established through the open cover theorem \cite{bao2024structure,bao2024unified,zhang2024stabilized}, hence a global estimate was lacking. This results that, in the numerical computation, one usually has to solve an optimization problem to obtain $k_0(\theta)$ first. Theorem~\ref{thm:minimal stab func well-defined} provides a global estimate for $k_0(\theta)$ and eliminates the need to first compute an approximate value of $k_0(\theta)$ in the practical applications.
\end{remark}

\subsection{Local energy estimate}

\begin{theorem}[local energy estimate]\label{thm:loc energy est} For any $\boldsymbol{p},\boldsymbol{q}\in\mathbb{R}^2\backslash\{\boldsymbol{0}\}$, let $\boldsymbol{p}=|\boldsymbol{p}|(\cos\varphi,\sin\varphi)^T,\boldsymbol{q}=|\boldsymbol{q}|(\cos\theta,\sin\theta)^T$, then for sufficiently large $k(\theta)$, 
    \begin{equation}\label{eqn:loc energy est}
        \frac{1}{|\boldsymbol{q}|}\Bigl(\hat{\boldsymbol{Z}}_k(\theta)\boldsymbol{p}\Bigr)\cdot(\boldsymbol{p}-\boldsymbol{q})\geq \hat{\gamma}(\varphi)|\boldsymbol{p}|-\hat{\gamma}(\theta)|\boldsymbol{q}|.
    \end{equation}
\end{theorem} 

\begin{proof}
    By the definition of $\hat{\boldsymbol{Z}}_k(\theta)$ in \eqref{def:surf energy mat}, we have
    \begin{equation}
        \begin{aligned}
            \frac{1}{|\boldsymbol{q}|}\Bigl(\hat{\boldsymbol{Z}}_k(\theta)\boldsymbol{p}\Bigr)\cdot\boldsymbol{p}&=\frac{|\boldsymbol{p}|^2}{|\boldsymbol{q}|}\hat{\boldsymbol{Z}}_k(\theta)\left(\begin{array}{c}
                 \cos\varphi  \\
                 \sin\varphi 
            \end{array}\right)\cdot\left(\begin{array}{c}
                 \cos\varphi  \\
                 \sin\varphi 
            \end{array}\right)\\
            &=\frac{|\boldsymbol{p}|^2}{|\boldsymbol{q}|}\left(\hat{\gamma}(\theta)-\hat{\gamma}^\prime(\theta)\sin 2(\theta-\varphi)+k(\theta)\sin^2(\theta-\varphi)\right)\\
            &=\frac{|\boldsymbol{p}|^2}{|\boldsymbol{q}|}P_k(\theta-\varphi,\theta).
        \end{aligned}
    \end{equation} \begin{equation}
        \begin{aligned}
            \frac{1}{|\boldsymbol{q}|}\Bigl(\hat{\boldsymbol{Z}}_k(\theta)\boldsymbol{p}\Bigr)\cdot\boldsymbol{q}&=|\boldsymbol{p}|\hat{\boldsymbol{Z}}_k(\theta)\left(\begin{array}{c}
                 \cos\varphi  \\
                  \sin\varphi
            \end{array}\right)\cdot\left(\begin{array}{c}
                 \cos\theta  \\
                  \sin\theta
            \end{array}\right)\\
            &=|\boldsymbol{p}|\left(\hat{\gamma}(\theta)\cos(\theta-\varphi)-\hat{\gamma}^\prime(\theta)\sin(\theta-\varphi)\right)\\
            &=|\boldsymbol{p}|\left(Q(\theta-\varphi,\theta)-\hat{\gamma}(\varphi)\right).
        \end{aligned}
    \end{equation} By Theorem~\ref{thm:minimal stab func well-defined}, for sufficiently large $k(\theta)\geq k_0(\theta)$, $4\hat{\gamma}(\theta)P_k(\theta-\varphi,\theta)\geq Q^2(\theta-\varphi,\theta)$. Therefore, \begin{equation}
        \begin{aligned}
            \frac{1}{|\boldsymbol{q}|}\Bigl(\hat{\boldsymbol{Z}}_k(\theta)\boldsymbol{p}\Bigr)\cdot(\boldsymbol{p}-\boldsymbol{q})&\geq \frac{|\boldsymbol{p}|^2}{4\hat{\gamma}(\theta)|\boldsymbol{q}|}Q^2(\theta-\varphi,\theta)-|\boldsymbol{p}|Q(\theta-\varphi,\theta)+|\boldsymbol{p}|\hat{\gamma}(\varphi)\\
            &\geq |\boldsymbol{p}|\hat{\gamma}(\varphi)-|\boldsymbol{q}|\hat{\gamma}(\theta).
        \end{aligned}
    \end{equation} The last inequality comes from the fact $\frac{1}{4a}t^2-t\geq -a,\,\,\forall t\in\mathbb{R}$.\qed
\end{proof}

\begin{remark}\label{rem:indep of Z}
    By taking $\boldsymbol{p}=-\boldsymbol{q}$ in \eqref{eqn:loc energy est}, i.e. $\theta=\varphi+\pi,|\boldsymbol{p}|=|\boldsymbol{q}|$. Then local energy estimate \eqref{eqn:loc energy est} gives $2\hat{\gamma}(\theta)\geq \hat{\gamma}(\theta-\pi)-\hat{\gamma}(\theta)$. This is consistent with condition \eqref{eqn:energy stab cond b}. This result can be extended to any surface energy matrix $\hat{\boldsymbol{Z}}(\theta)$ satisfying $\begin{pmatrix} \cos\theta & \sin\theta \end{pmatrix}\hat{\boldsymbol{Z}}(\theta)\begin{pmatrix} \cos\theta \\ \sin\theta \end{pmatrix}=\hat{\gamma}(\theta)$. Therefore, within the framework of using local energy estimate, the energy stability condition \eqref{eqn:energy stab cond b} is considered optimal and cannot be further improved. 
\end{remark}

\subsection{Unconditional energy stability}

\begin{proof}
    Suppose $k(\theta)$ is sufficiently large, satisfying $k(\theta)\geq k_0(\theta)$. 
    
    For any $m\geq 0$, we have \begin{equation}\label{eqn:energy difference 1}
        \begin{aligned}
            &\left(\hat{\boldsymbol{Z}}_k(\theta^m)\partial_s\boldsymbol{X}^{m+1},\partial_s(\boldsymbol{X}^{m+1}-\boldsymbol{X}^m)\right)_{\Gamma^m}^h\\
            &=\sum_{j=1}^N \left[|\boldsymbol{h}_j^m|\left(\hat{\boldsymbol{Z}}_k(\theta_j^m)\frac{\boldsymbol{h}_j^{m+1}}{|\boldsymbol{h}_j^m|}\right)\cdot\frac{\boldsymbol{h}_j^{m+1}-\boldsymbol{h}_j^m}{|\boldsymbol{h}_j^m|}\right]\\
            &=\sum_{j=1}^N\left[\frac{1}{|\boldsymbol{h}^m_j|}\left(\hat{\boldsymbol{Z}}_k(\theta_j^m)\boldsymbol{h}_j^{m+1}\right)\cdot(\boldsymbol{h}_j^{m+1}-\boldsymbol{h}_j^m)\right]
        \end{aligned}
    \end{equation}

    Combining local energy estimate \eqref{eqn:loc energy est} with equation \eqref{eqn:energy difference 1} gives \begin{equation}\label{eqn:energy difference 2}
        \begin{aligned}
            &\left(\hat{\boldsymbol{Z}}_k(\theta^m)\partial_s\boldsymbol{X}^{m+1},\partial_s(\boldsymbol{X}^{m+1}-\boldsymbol{X}^m)\right)_{\Gamma^m}^h\\
            &\geq\sum_{j=1}^N\left[|\boldsymbol{h}_j^{m+1}|\hat{\gamma}(\theta_j^{m+1})-|\boldsymbol{h}_j^m|\hat{\gamma}(\theta_j^m)\right]\\
            &=\sum_{j=1}^N|\boldsymbol{h}_j^{m+1}|\hat{\gamma}(\theta_j^{m+1})-\sum_{j=1}^N|\boldsymbol{h}_j^m|\hat{\gamma}(\theta_j^m)=W_c^{m+1}-W_c^m.
        \end{aligned}
    \end{equation}
    
    By taking $\varphi^h=\mu^{m+1},\boldsymbol{\omega}^h=\boldsymbol{X}^{m+1}-\boldsymbol{X}^m$ in \eqref{eqn:SP-PFEM}, we have \begin{equation}
        \begin{aligned}
            \Bigl(\hat{\boldsymbol{Z}}_k(\theta^m)\partial_s\boldsymbol{X}^{m+1},\partial_s(\boldsymbol{X}^{m+1}-\boldsymbol{X}^m)\Bigr)^h_{\Gamma^m}&=\Bigl(\mu^{m+1}\boldsymbol{n}^{m+\frac{1}{2}},\boldsymbol{X}^{m+1}-\boldsymbol{X}^m\Bigr)^h_{\Gamma^m}\\
            &=-\tau\Bigl(\partial_s\mu^{m+1},\partial_s\mu^{m+1}\Bigr)^h_{\Gamma^m}.
        \end{aligned}
    \end{equation} Together with \eqref{eqn:energy difference 2} yields \begin{equation}
            W_c^{m+1}-W_c^m\leq -\tau\Bigl(\partial_s\mu^{m+1},\partial_s\mu^{m+1}\Bigr)^h_{\Gamma^m}\leq 0,\qquad\forall m\geq 0.
    \end{equation} This completes the proof of Theorem~\ref{thm:structure-preserving}.\qed
\end{proof}

\section{Upper bound of the minimal stabilizing function}\label{sec:existence of k0}

To establish the existence of $k_0(\theta)$ with the optimal energy stability condition \eqref{eqn:energy stab cond b}, a very sharp estimate is required for $Q(\phi,\theta)$. Comparing to other estimates in the literature \cite{bao2024structure,zhang2024stabilized}, this estimate explicitly relates to both $P_\alpha(\phi,\theta)$ and $\hat{\gamma}(\theta)$. Lemma~\ref{lma:tight estimate lemma} provides the crucial estimate, which is essential for our existence proof. We start with the following lemma to explore the properties of the optimal energy stability condition \eqref{eqn:energy stab cond b}.

\begin{remark}
    In \cite{zhang2024stabilized}, the bound of $Q(\phi,\theta)$ is controlled by only concerning with $\hat{\gamma}(\theta)$. In \cite{bao2024structure}, the authors establish the required estimates by coupling $P_\alpha(\phi,\theta)$ and $Q(\phi,\theta)$ with $\hat{\gamma}(\theta)$ separately. Therefore, the estimates they obtained are relatively less refined. Here, we explicitly combine $P_\alpha,\hat{\gamma}(\theta)$ and $Q$ to obtain a sharper estimate.
\end{remark}

\begin{lemma}\label{lma:tight estimate lemma}
    Let $f$ be a non-negative $C^2$ function on $2\pi\mathbb{T}$. Then for any positive constant $C\geq \sup\limits_{2\pi\mathbb{T}}|f^{\prime\prime}|$, we have \begin{equation} \label{eqn:sharp estimate lemma}
        |f^\prime(x) y| \leq f(x) + \frac{C}{2}y^2, \qquad \forall x,y\in 2\pi\mathbb{T}.
    \end{equation} 
\end{lemma}

\begin{proof}Since $f$ is $C^2$ defined on $2\pi\mathbb{T}$, we know $f^{\prime\prime}$ is bounded. By the mean value theorem and the non-negativity of $f$, for any positive constant $C\geq \sup\limits_{2\pi\mathbb{T}}|f^{\prime\prime}|$, it holds 
     \begin{equation}\label{eqn:poly of y}
        \begin{aligned}
            0&\leq f(x+y)\leq f(x)+f^{\prime}(x)y+\frac{C}{2}y^2,\qquad\forall x, y\in 2\pi\mathbb{T}.
        \end{aligned}
    \end{equation} 
    Therefore, we know $-f'(x)y\leq f(x)+\frac{C}{2}y^2$ and $f'(x)y\leq f(x)+\frac{C}{2}y^2$, which implies \eqref{eqn:sharp estimate lemma}.\qed
\end{proof}

\begin{remark}
    Lemma~\ref{lma:tight estimate lemma} plays a crucial role in bounding $Q(\phi,\theta)$ and analyzing the critical situation when $\phi=\pi$. It's challenging to obtain similar inequalities in the $\gamma(\boldsymbol{n})$ formulation.
\end{remark}

\begin{remark}
    If attempting to obtain an inequality similar to Lemma~\ref{lma:tight estimate lemma} in the $\gamma(\boldsymbol{n})$ formulation, thorny difficulties may arise. Suppose $\gamma(\boldsymbol{n})$ is expanded at $\boldsymbol{n}_0$, when the line connecting $\boldsymbol{n},\boldsymbol{n}_0$ passes through the origin, the Hessian matrix $\boldsymbol{H}_\gamma$ becomes unbounded as it has no definition at $\boldsymbol{0}$. This prevents its gradient $\nabla\gamma(\boldsymbol{n})$ from being effectively controlled. Even if a similar estimate can be obtained, the corresponding coefficient $C$ would depend on $\boldsymbol{n}_0$, rather than being a constant as in the case of $\hat{\gamma}(\theta)$ formulation.
\end{remark}

\begin{lemma}[Estimation of $Q(\phi,\theta)$]\label{lma:est of Q}
    Suppose the optimal energy stability condition \eqref{eqn:energy stab cond b} holds. For $Q(\phi,\theta)$ defined in \eqref{eqn:auxiliary func Q}. Then \begin{equation}
        |Q(\phi,\theta)|\leq|P_{A(\theta)}(\phi,\theta)+\hat{\gamma}(\theta)|,\qquad\forall\phi\in2\pi\mathbb{T},
    \end{equation}
    where $A(\theta)$ is defined in \eqref{eqn:A theta}.
\end{lemma} 

\begin{proof}Firstly, we notice that the lower bound of $Q(\phi,\theta)$ can be obtained by 
    \begin{align}
        Q(\phi,\theta)+P_0(\phi,\theta)+\hat{\gamma}(\theta)&=\hat{\gamma}(\theta-\phi)+ \hat{\gamma}(\theta)(2+\cos\phi)-\hat{\gamma}^\prime(\theta)(\sin\phi+\sin 2\phi)\nonumber\\
        &\geq\hat{\gamma}(\theta-\phi)+\hat{\gamma}(\theta)-\hat{\gamma}^\prime(\theta)(1+2\cos\phi)\sin\phi\nonumber\\
        &\geq \hat{\gamma}(\theta)-3|\hat{\gamma}^\prime(\theta)||\sin\phi|\nonumber\\
        &\geq - \frac{9}{2}\sup\limits_{2\pi\mathbb{T}} | \hat{\gamma}^{\prime\prime}|\, \sin^2\phi,
    \end{align}
where the last inequality comes from \eqref{eqn:poly of y}. Furthermore, using the fact that $A(\theta)\geq \frac{9}{2}\sup\limits_{2\pi\mathbb{T}} | \hat{\gamma}^{\prime\prime}|$ and $P_{A(\theta)}(\phi,\theta) = P_0(\phi,\theta)+A(\theta)\sin^2\phi$, we have \begin{equation}
    Q(\phi,\theta)\geq -P_{A(\theta)}(\phi,\theta)-\hat{\gamma}(\theta),\qquad\forall\phi\in 2\pi\mathbb{T}.
\end{equation}

For the other direction of the inequality, we first observe that $Q(0, \theta) - P_0(0,\theta)-\hat{\gamma}(\theta) = 0$, and $Q(\pi, \theta) - P_0(\pi,\theta)-\hat{\gamma}(\theta) = \hat{\gamma}(\theta-\pi)-3\hat{\gamma}(\theta)\leq 0$. Thus we divide it into the following two cases:

Case 1: For $|\phi|\leq\frac{\pi}{2}$. Apply the mean value theorem to $Q(\cdot, \theta) - P_0(\cdot,\theta)-\hat{\gamma}(\theta)$ on $[\phi, 0]$, we know there exists a $\xi\in [\phi, 0], |\xi|\leq \frac{\pi}{2}$ such that
 \begin{align}\label{eqn:case 1}
    &Q(\phi,\theta)-P_0(\phi,\theta)-\hat{\gamma}(\theta)\nonumber\\
    &=\frac{1}{2} \left(\hat{\gamma}^{\prime\prime}(\theta-\xi) - \hat{\gamma}(\theta)\cos\xi + \hat{\gamma}^\prime(\theta)\sin\xi - 4\hat{\gamma}^\prime(\theta)\sin 2\xi\right)\phi^2\nonumber\\
    &\leq \frac{1}{2} \left(\sup_{2\pi\mathbb{T}}|\hat{\gamma}^{\prime\prime}|+5|\hat{\gamma}^\prime(\theta)| + \hat{\gamma}(\theta)\right)\left(\frac{\pi^2}{4}\sin^2\phi\right)\\
    &\leq A(\theta)\sin^2\phi.\nonumber
\end{align}

    Case 2: For $|\phi-\pi|<\frac{\pi}{2}$. By the condition \eqref{eqn:energy stab cond b}, we know $3\hat{\gamma}(\theta)-\hat{\gamma}(\theta-\pi)$ is a non-negative $C^2$ function. Using Lemma~\ref{lma:tight estimate lemma} to $\hat{\gamma}(\theta-\pi)-3\hat{\gamma}(\theta)$, we have 
    \begin{align*}
        &\left|\left(3\hat{\gamma}^\prime(\theta)-\hat{\gamma}^\prime(\theta-\pi)\right) (\phi - \pi)\right|\\
        &\leq \left(3\hat{\gamma}(\theta) - \hat{\gamma}(\theta-\pi)\right) + \frac{\sup\limits_{\theta\in 2\pi\mathbb{T}}\left|3\hat{\gamma}^{\prime\prime}(\theta) - \hat{\gamma}^{\prime\prime}(\theta-\pi)\right|}{2}(\phi-\pi)^2\\
        &\leq \left(3\hat{\gamma}(\theta) - \hat{\gamma}(\theta-\pi)\right) + 2\sup\limits_{2\pi\mathbb{T}}|\hat{\gamma}^{\prime\prime}|(\phi-\pi)^2, \quad \forall |\phi-\pi|< \frac{\pi}{2}.
    \end{align*}

    Apply the mean value theorem to $Q(\cdot, \theta) - P_0(\cdot,\theta)-\hat{\gamma}(\theta)$ on $[\phi, \pi]$, there exists a $\xi\in [\phi, \pi]$ with $|\xi-\pi|< \frac{\pi}{2}$ such that 
    \begin{align}\label{eqn:case 2}
        &Q(\phi,\theta)-P_0(\phi,\theta)-\hat{\gamma}(\theta)\nonumber\\
        & = \left(\hat{\gamma}(\theta-\pi)-3\hat{\gamma}(\theta)\right)+ \left(-\hat{\gamma}^\prime(\theta-\pi) + 3 \hat{\gamma}^\prime(\theta)\right)(\phi-\pi)\nonumber\\
        & \quad + \frac{1}{2}\left(\hat{\gamma}^{\prime\prime}(\theta-\xi) - \hat{\gamma}(\theta)\cos\xi + \hat{\gamma}^\prime(\theta)\sin\xi - 4\hat{\gamma}^\prime(\theta)\sin 2\xi\right)(\phi-\pi)^2\nonumber\\
        & \leq \left(\hat{\gamma}(\theta-\pi)-3\hat{\gamma}(\theta)\right)+ \left(3\hat{\gamma}(\theta) - \hat{\gamma}(\theta-\pi)\right)\nonumber\\
        & \quad + \frac{1}{2}\left(\sup_{2\pi\mathbb{T}}|\hat{\gamma}^{\prime\prime}| + 5|\hat{\gamma}^\prime(\theta)|+ \hat{\gamma}(\theta) + 4\hat{\gamma}^\prime(\theta)\sin 2\xi + 4\sup_{2\pi\mathbb{T}}|\hat{\gamma}^{\prime\prime}|\right)(\phi-\pi)^2\nonumber\\
        & \leq \frac{\pi^2}{8}\left(5\sup_{2\pi\mathbb{T}}|\hat{\gamma}^{\prime\prime}| + 5|\hat{\gamma}^\prime(\theta)|+ \hat{\gamma}(\theta) \right)\sin^2\left(\phi-\pi\right) = A(\theta)\sin^2\phi.
    \end{align}
    
    Combining \eqref{eqn:case 1} with \eqref{eqn:case 2} yields \begin{equation}
        Q(\phi,\theta)\leq P_0(\phi,\theta)+\hat{\gamma}(\theta)+A(\theta)\sin^2\phi = P_{A(\theta)}(\phi,\theta)+\hat{\gamma}(\theta),\quad\forall\phi\in 2\pi\mathbb{T}.
    \end{equation} This completes the proof.\qed
\end{proof}

With the help of the sharp estimate given in Lemma~\ref{lma:est of Q}, Theorem~\ref{thm:minimal stab func well-defined} is ready to be proven. 

\begin{proof}\textit{(Existence of the minimal stabilizing function)}
    By Lemma~\ref{lma:est of Q}, we have \begin{equation}
       Q^2(\phi,\theta)\leq \left(P_{A(\theta)}(\phi,\theta)+\hat{\gamma}(\theta)\right)^2
   \end{equation} Recall the definition of $P_\alpha(\phi,\theta)$ in \eqref{eqn:auxiliary func P}, we have $P_\alpha(\phi,\theta) = P_0(\phi,\theta)+ \alpha\sin^2\phi$, and further
   \begin{equation}
       \begin{aligned}
           &4\hat{\gamma}(\theta)P_\alpha(\phi,\theta)-Q^2(\phi,\theta)\\
           &\geq 4\hat{\gamma}(\theta)P_{A(\theta)}(\phi,\theta)-\left(P_{A(\theta)}(\phi,\theta)+\hat{\gamma}(\theta)\right)^2 + 4\hat{\gamma}(\theta)\left(\alpha-A(\theta)\right)\sin^2\phi\\
           & = -\left(P_{A(\theta)}(\phi,\theta)-\hat{\gamma}(\theta)\right)^2 + 4\hat{\gamma}(\theta)\left(\alpha-A(\theta)\right)\sin^2\phi\\
           & = \left(- \left(-2\hat{\gamma}^\prime(\theta)\cos \phi + A(\theta)\sin \phi\right)^2 + 4\hat{\gamma}(\theta)\left(\alpha-A(\theta)\right)\right)\sin^2\phi\\
           &\geq \left[4\hat{\gamma}(\theta)\alpha-A^2(\theta)-4|\hat{\gamma}^\prime(\theta)|^2-4\hat{\gamma}(\theta)A(\theta)\right]\sin^2\phi.
       \end{aligned}
   \end{equation} 
    The last inequality follows from the fact that $|a\cos\phi + b\sin\phi|\leq \sqrt{a^2+b^2}$.

   Therefore, for any $\alpha\geq \frac{1}{4\hat{\gamma}(\theta)}\left[A^2(\theta)+4\hat{\gamma}(\theta)A(\theta)+4|\hat{\gamma}^{\prime}(\theta)|^2\right]$, we have \begin{equation*}
       4\hat{\gamma}(\theta)P_\alpha(\phi,\theta)-Q^2(\phi,\theta)\geq 0,\qquad\forall\phi\in 2\pi\mathbb{T}.
   \end{equation*}

   Which implies that \begin{equation}
       k_0(\theta)\leq \frac{1}{4\hat{\gamma}(\theta)}\left[A^2(\theta)+4\hat{\gamma}(\theta)A(\theta)+4|\hat{\gamma}^{\prime}(\theta)|^2\right]<\infty.
   \end{equation}\qed
\end{proof}

\section{Numerical results}\label{sec:numer}

In this section, numerical experiments are presented to demonstrate the high performance of the proposed SP-PFEM. We illustrate the efficiency/accuracy using a convergence test, and verify the structure-preserving properties of the proposed method, i.e. area conservation and unconditional energy stability.

In the convergence tests, the following two types of anisotropies are considered: \begin{itemize}
    \item Case I: $\hat{\gamma}(\theta)=1+\beta\cos 3\theta$ with $|\beta|<1$. It is weakly anisotropic when $|\beta|<\frac{1}{8}$ and strongly anisotropic otherwise;\\
    \item Case II: $\hat{\gamma}(\theta)=\sqrt{\left(\frac{5}{2}+\frac{3}{2}\,\text{sgn}(n_1)\right)n_1^2+n_2^2}$, here $(n_1,n_2)^T=(-\sin\theta,\cos\theta)^T$.
\end{itemize}

The schemes formally have the quadratic convergence rate in space and linear convergence rate in time, so the uniform time step $\tau$ is chosen as $\tau=h^2$, unless it is stated otherwise. 

We adopt the manifold distance \cite{zhao2021energy,li2021energy} \begin{equation}
    M(\Gamma_1,\Gamma_2)\coloneqq |(\Omega_1\backslash\Omega_2)\cup(\Omega_2\backslash\Omega_1)|=2|\Omega_1\cup\Omega_2|-|\Omega_1|-|\Omega_2|,
\end{equation} to measure the distance between two curves $\Gamma_1,\Gamma_2$, where $\Omega_i (i=1,2)$ represent the interior regions of $\Gamma_i$ and $|\Omega_i|$ denotes its area.

Let $\Gamma^m$ be the numerical approximation of $\Gamma^h(t=t_m\coloneqq m\tau)$, the numerical error is given as \begin{equation}
    e^h(t)\Bigl|_{t=t_m}\coloneqq M(\Gamma^m,\Gamma(t=t_m)).
\end{equation} Since the exact solution cannot be obtained analytically, we numerically approximated $\Gamma(t=t_m)$ using fine meshes $h_e=2^{-8},\tau_e=2^{-16}$ in \eqref{eqn:SP-PFEM}.

Following indicators are introduced to numerically demonstrate mesh quality, area conservation and energy stability: the weighted mesh ratio \begin{equation}
    R_{\gamma}^h(t)\coloneqq \frac{\max\limits_{1\leq j\leq N}\hat{\gamma}(\theta_j)|\boldsymbol{h}_j|}{\min\limits_{1\leq j\leq N}\hat{\gamma}(\theta_j)|\boldsymbol{h}_j|},
\end{equation} the normalized area loss and the normalized energy for closed curves, \begin{equation}
    \left.\frac{\Delta A^h_c(t)}{A^h_c(0)}\right|_{t=t_m}\coloneqq \frac{A^m_c-A^0_c}{A_c^0},\qquad\left.\frac{W^h_c(t)}{W^h_c(0)}\right|_{t=t_m}\coloneqq \frac{W^m_c}{W^0_c}.
\end{equation}

In the following simulations, the initial shapes are always chosen as an ellipse with major axis 4 and minor axis 1, except it is stated otherwise. In Newton's iteration, the tolerance value is set to be $\text{tol}=10^{-12}$.

The minimal stabilizing function $k_0(\theta)$ is obtained as follows: we solve the optimization problem \eqref{eqn:minimal stab func} for $\theta_j=-\pi+j\frac{\pi}{10},0\leq j\leq 20$ to determine $k_0(\theta_j)$, then do linear interpolation for the intermediate points. 

We first validate the upper bound of $k_0(\theta)$ provided in Theorem~\ref{thm:minimal stab func well-defined} (cf. Fig.~\ref{fig:upperbound}). Here, $K(\theta)$ denotes the upper bound given in \eqref{eqn:upper bound}.

\begin{figure}[htbp]
    \centering
    \includegraphics[width=1\textwidth]{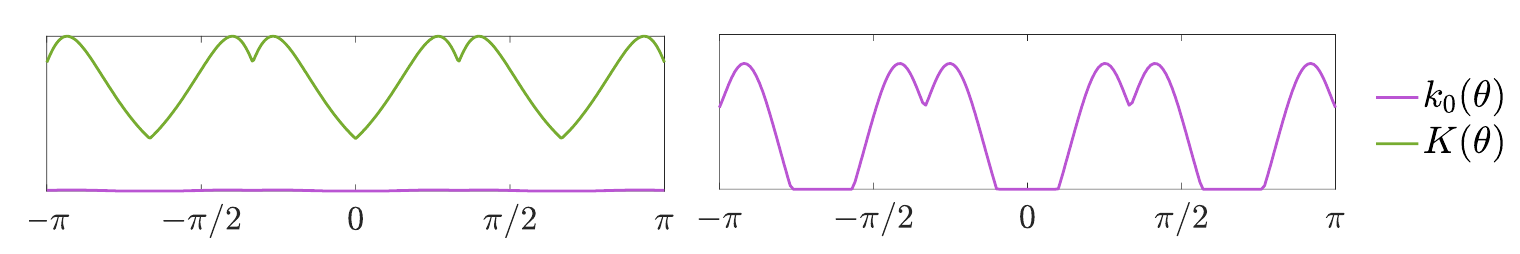}
    \caption{Minimal stabilizing function $k_0(\theta)$ and upper bound $K(\theta)$ in \eqref{eqn:upper bound} for Case I with $\beta=\frac{1}{2}$.}
    \label{fig:upperbound}
\end{figure}

\begin{figure}[htbp]
    \centering
    \includegraphics[width=0.5\textwidth]{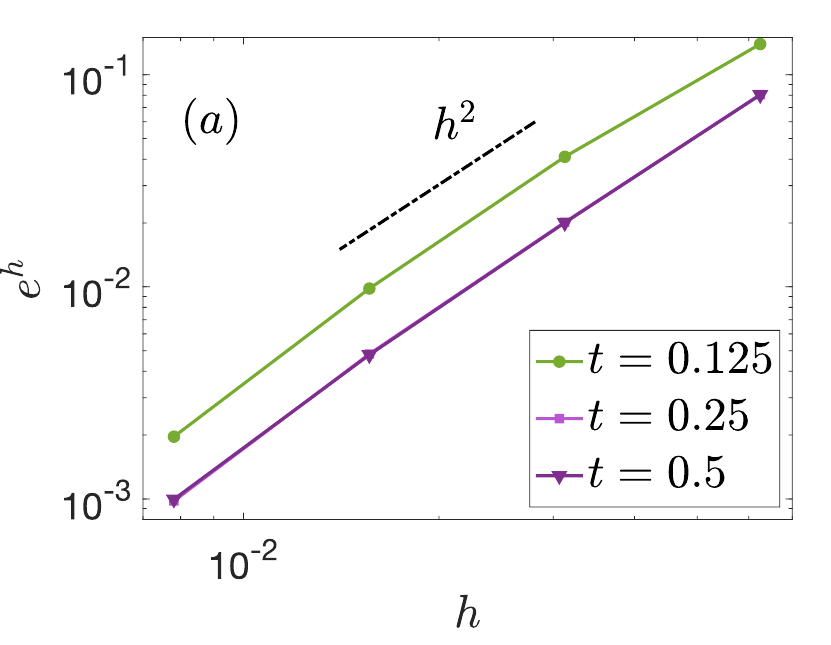}\includegraphics[width=0.5\textwidth]{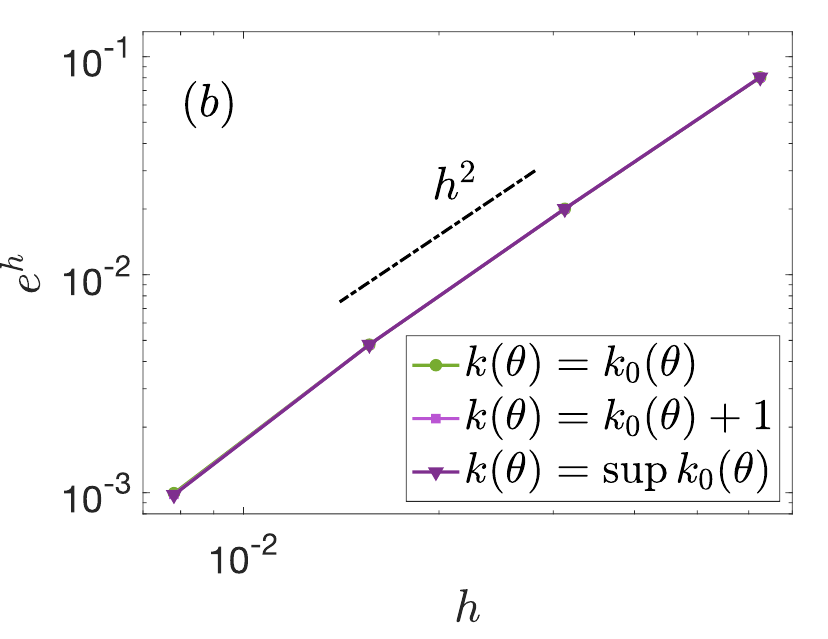}\\\includegraphics[width=0.5\textwidth]{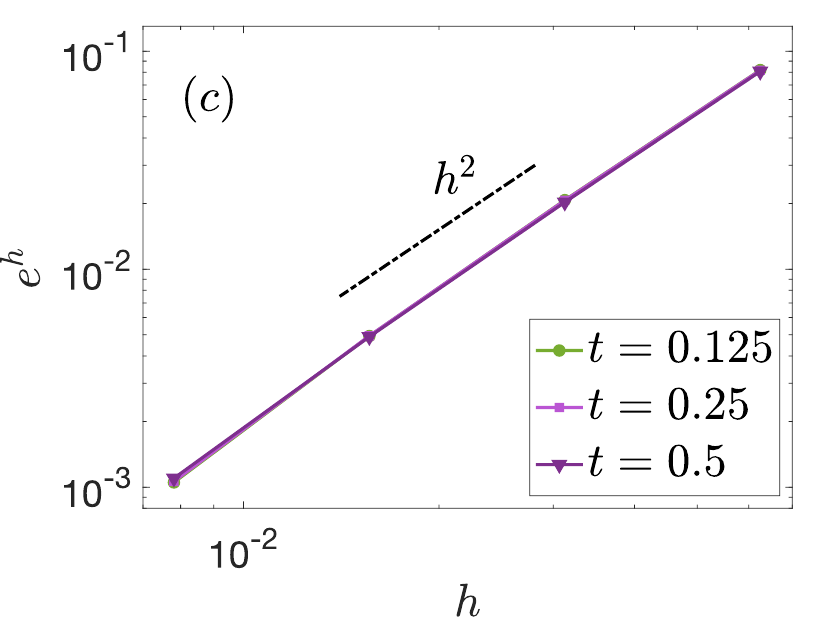}\includegraphics[width=0.5\textwidth]{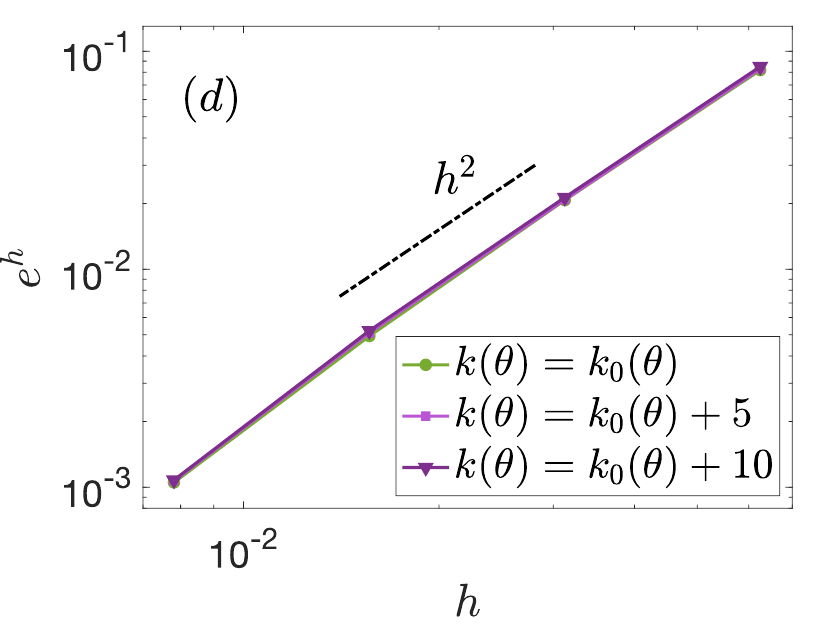}
    \caption{Convergence rates of the SP-PFEM \eqref{eqn:SP-PFEM} for Case I with $\beta=1/9$ (a) at different times with $k(\theta)=k_0(\theta)$, and (b) at $t=0.5$ with different $k(\theta)$; and for Case II (c) at different times with $k(\theta)=k_0(\theta)$, and (d) at $t=0.5$ with different $k(\theta)$.}
    \label{fig:convergent}
\end{figure}

\begin{figure}[htbp]
    \centering
    \includegraphics[width=0.5\textwidth]{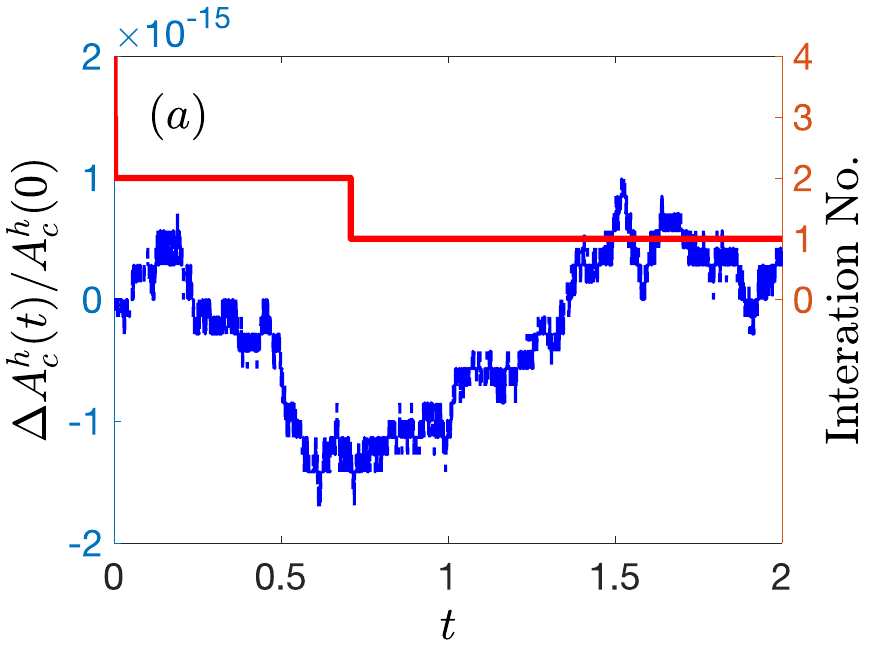}\includegraphics[width=0.5\textwidth]{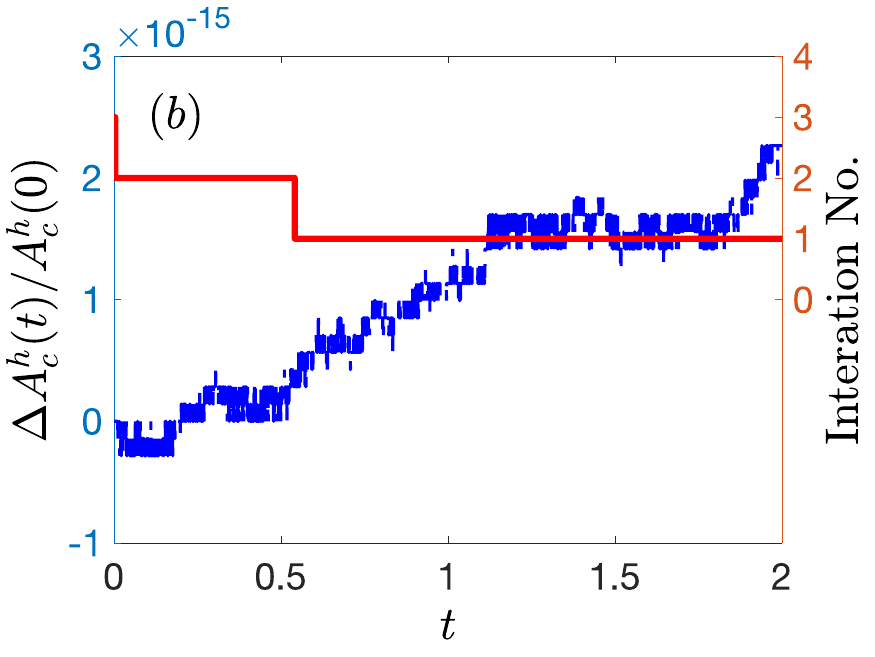}
    \caption{Normalized area loss (blue dash line) and iteration number (red line) of the SP-PFEM \eqref{eqn:SP-PFEM} with $k(\theta)=k_0(\theta)$ and $h=2^{-6},\tau=2^{-12}$ for (a) Case I with $\beta=1/2$; and for (b) Case II.}
    \label{fig:volume}
\end{figure}
    
\begin{figure}[htbp]
    \centering
    \includegraphics[width=0.5\textwidth]{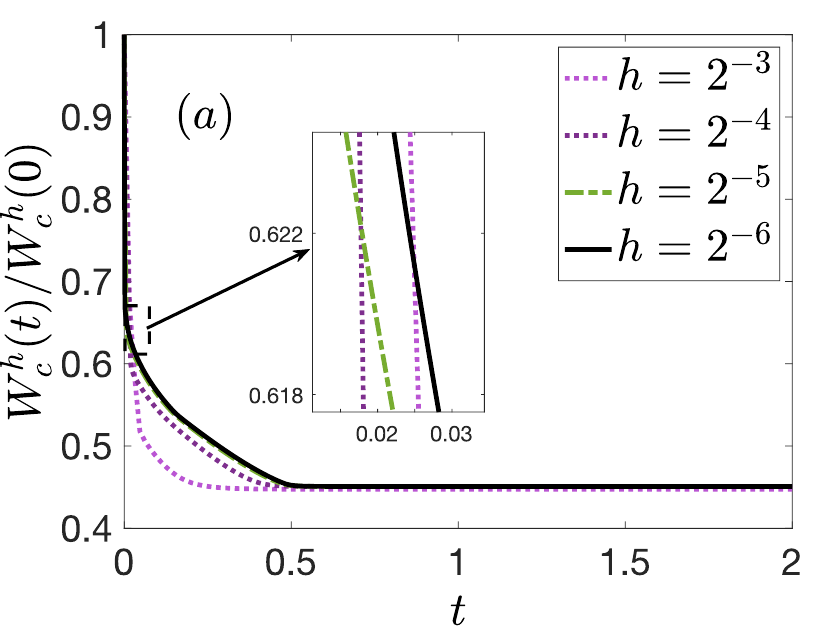}\includegraphics[width=0.5\textwidth]{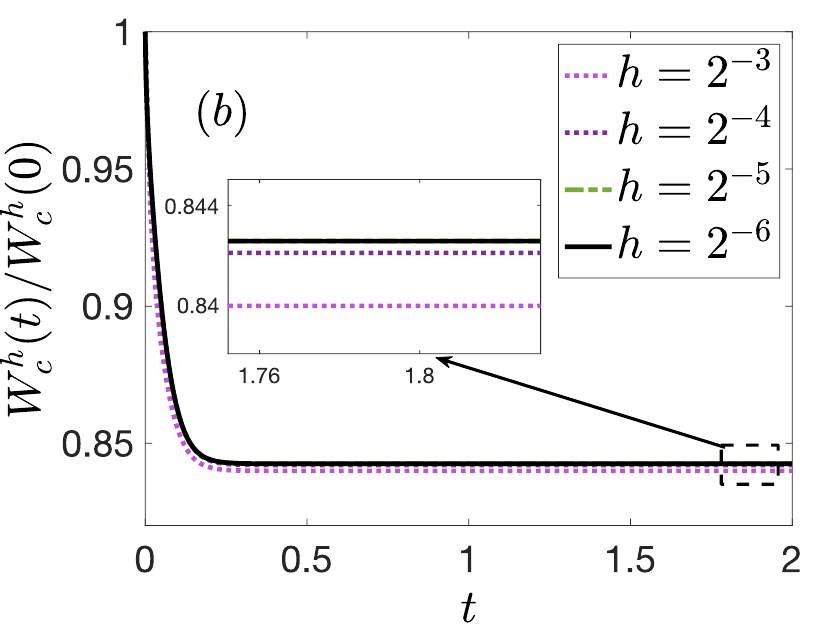}
    \caption{Normalized energy of the SP-PFEM \eqref{eqn:SP-PFEM} with $k(\theta)=k_0(\theta)$ for (a) Case I with $\beta=1/2$; and for (b) Case II.}
    \label{fig:energy}
\end{figure}

\subsection{Efficiency, accuracy and structure-preserving property}

It can be observed from Fig.~\ref{fig:convergent}--Fig.~\ref{fig:mesh} that: \begin{itemize}
    \item The SP-PFEM \eqref{eqn:SP-PFEM} possesses second-order spatial and first-order temporal accuracy (cf. Fig.~\ref{fig:convergent}).
    \item The normalized area loss is around $10^{-15}$, matching the order of the round-off error (cf. Fig.~\ref{fig:volume}). Thus the area is conserved up to the machine precision.
    \item Numbers of Newton's iteration descend to $2$ in a very short time, and finally $1$. This observation suggests that the fully implicit scheme can be solved with high computational efficiency (cf. Fig.~\ref{fig:volume}).
    \item The normalized energy is monotonically decreasing when $\hat{\gamma}(\theta)$ satisfies condition \eqref{eqn:energy stab cond b} (cf. Fig.~\ref{fig:energy}). For Case I, condition \eqref{eqn:energy stab cond b} requires $0<\beta\leq \frac{1}{2}$, Fig.~\ref{fig:energy} (a) shows that the proposed method \eqref{eqn:SP-PFEM} still guarantees energy dissipation when $\beta$ takes its maximum, confirming the conclusion of Fig.~\ref{thm:structure-preserving}. In contrast to the conclusion in \cite{bao2023symmetrized2D}, \eqref{eqn:SP-PFEM} also exhibits unconditional energy stability for asymmetric surface energies.
    \item The weighted mesh ratio $R^h_{\gamma}(t)$ tends to a constant as $t\to+\infty$, which suggests an asymptotic quasi-uniform mesh distribution of \eqref{eqn:SP-PFEM} (cf. Fig.~\ref{fig:mesh}).
\end{itemize}

\begin{figure}[htbp]
    \centering
    \includegraphics[width=0.5\textwidth]{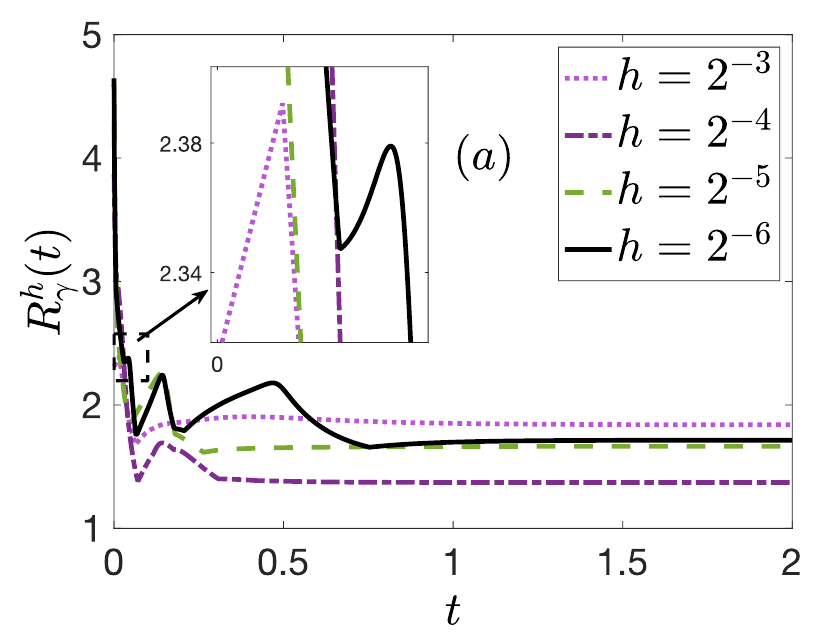}\includegraphics[width=0.5\textwidth]{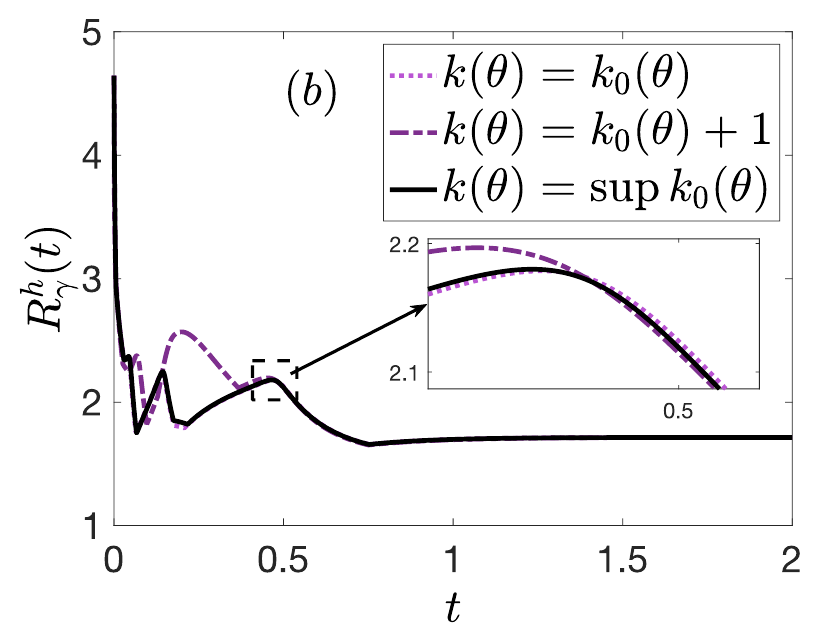}\\\includegraphics[width=0.5\textwidth]{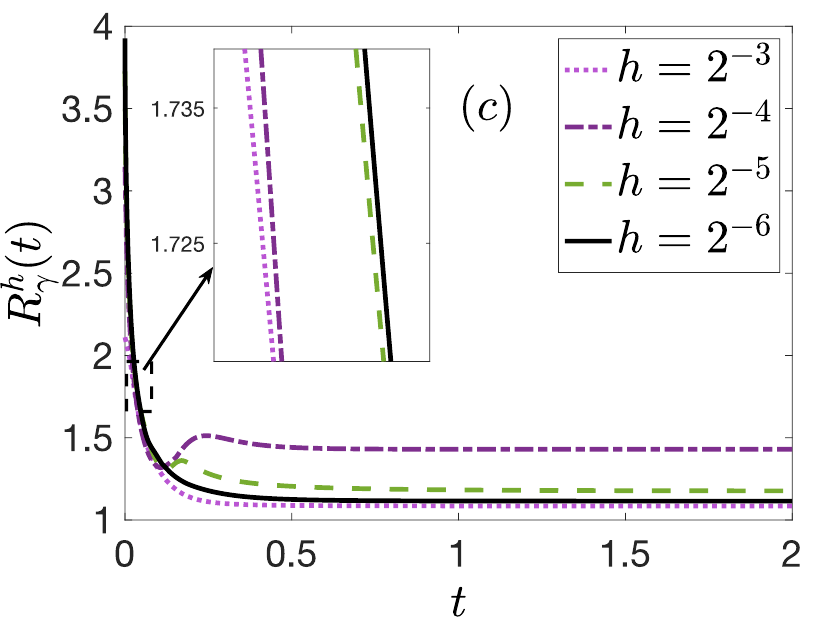}\includegraphics[width=0.5\textwidth]{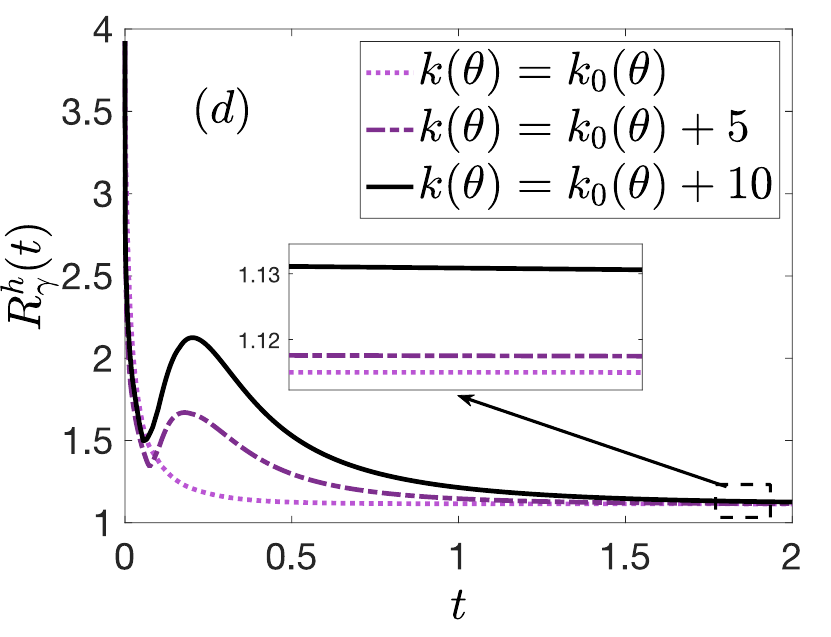}
    \caption{Weighted mesh ratio of the SP-PFEM \eqref{eqn:SP-PFEM} for Case I with $\beta=1/9$ is illustrated in (a) for $k(\theta)=k_0(\theta)$ with different $h$, and in (b) for different $k(\theta)$ with $h=2^{-6}$; and for Case II is presented in (c) for $k(\theta)=k_0(\theta)$ with different $h$, and in (d) for different $k(\theta)$ with $h=2^{-6}$.}
    \label{fig:mesh}
\end{figure}

\subsection{Application for morphological evolutions}

In the following we apply the SP-PFEM \eqref{eqn:SP-PFEM} to simulate the morphological evolutions of closed curves under anisotropic surface diffusion. The mesh size is chosen as $h=2^{-6},\tau=h^2$.\\

Fig.~\ref{fig:morphological evo} plots the morphological evolutions of an ellipse with major axis $4$ and minor axis $1$ under anisotropic surface diffusion with four different surface energies: \begin{enumerate}[label=(\alph*)]
    \item the $3$-fold anisotropy: $\hat{\gamma}(\theta)=1+\frac{1}{2}\cos 3\theta$;
    \item the piecewise BGN-anisotropy: $\hat{\gamma}(\theta)=\sqrt{\left(\frac{5}{2}+\frac{3}{2}\,\text{sgn}(n_1)\right)n_1^2+n_2^2}$, with $\boldsymbol{n}=(n_1,n_2)^T\coloneqq(-\sin\theta,\cos\theta)^T$;
    \item the $4$-fold anisotropy: $\hat{\gamma}(\theta)=1+\frac{1}{10}\cos 4\theta$;
    \item the regularized crystalline anisotropy: $\hat{\gamma}(\theta)=1+\sqrt{\varepsilon^2+(1-\varepsilon^2)\sin^2\frac{m}{2}\theta}$, with $\varepsilon=0.1,m=7$.
\end{enumerate}

\begin{figure}[t!]
    \centering
    \includegraphics[width=0.98\textwidth]{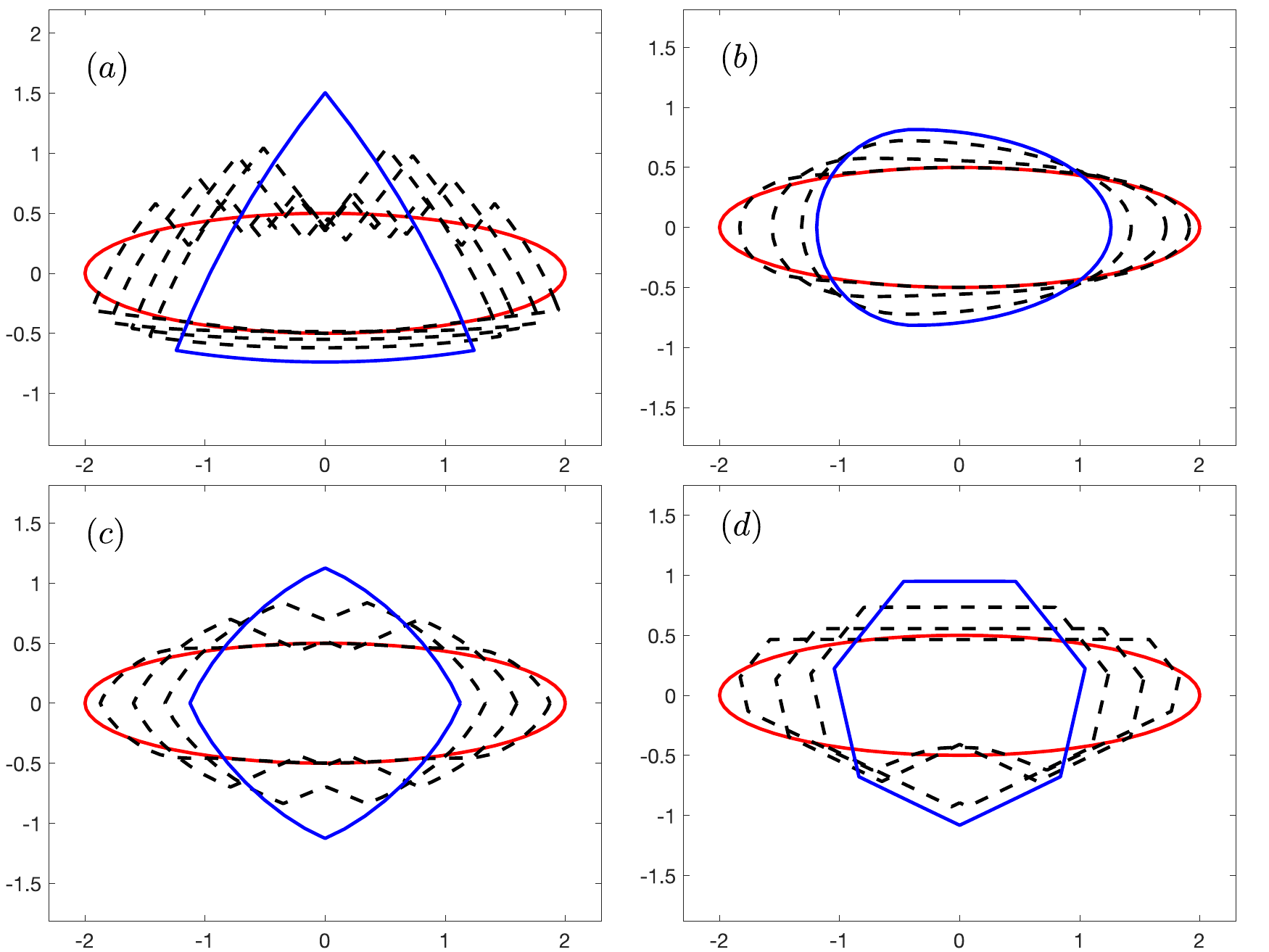}
    \caption{Morphological evolutions of an ellipse with major axis 4 and minor axis 1 under anisotropic surface diffusion with different surface energies: (a) Case I with $\beta=1/2$; (b) Case II; (c) the $4$-fold anisotropy $\hat{\gamma}(\theta)=1+\frac{1}{10}\cos 4\theta$; (d) the regularized crystalline anisotropy $\hat{\gamma}(\theta)=$ $\hat{\gamma}(\theta)=1+\sqrt{\varepsilon^2+(1-\varepsilon^2)\sin^2\frac{m}{2}\theta}$, with $\varepsilon=0.1,m=7$. The red and blue lines represent the initial shape and the numerical equilibrium, respectively; and the black dash lines represent the intermediate curves}
    \label{fig:morphological evo}
\end{figure}

In the above surface energies, anisotropy (c) is symmetric while others are asymmetric. In can be observed from Fig.~\ref{fig:morphological evo}, compared to the symmetrized SP-PFEM in \cite{bao2023symmetrized2D}, the proposed SP-PFEM \eqref{eqn:SP-PFEM} is not only applicable to symmetric surface energies, but also to asymmetric surface energies. Moreover, Fig.~\ref{fig:morphological evo} (b) indicates that the proposed method also performs effectively for $\hat{\gamma}(\theta)$ with lower regularity, i.e. globally $C^1$ and piecewise $C^2$. It demonstrates improved performance across a broader range of surface energies.\\

\subsection{Multiple equilibria for open curves in solid-state dewetting}

Following similar derivations in \cite{li2021energy,li2023symmetrized,zhang2024stabilized}, the SP-PFEM \eqref{eqn:SP-PFEM} can also be extended to simulate the morphological evolutions of open curves in solid-state dewetting. The numerical scheme is analogous to \cite[(5.3)]{li2023symmetrized} or \cite[(6.6)]{zhang2024stabilized}. Details are omitted here for brevity. The same as Theorem~\ref{thm:structure-preserving}, the corresponding method is area conservative and proven to be unconditionally energy-stable under \eqref{eqn:energy stab cond b}.

\begin{figure}[htbp]
    \centering
    \includegraphics[width=1.0\textwidth]{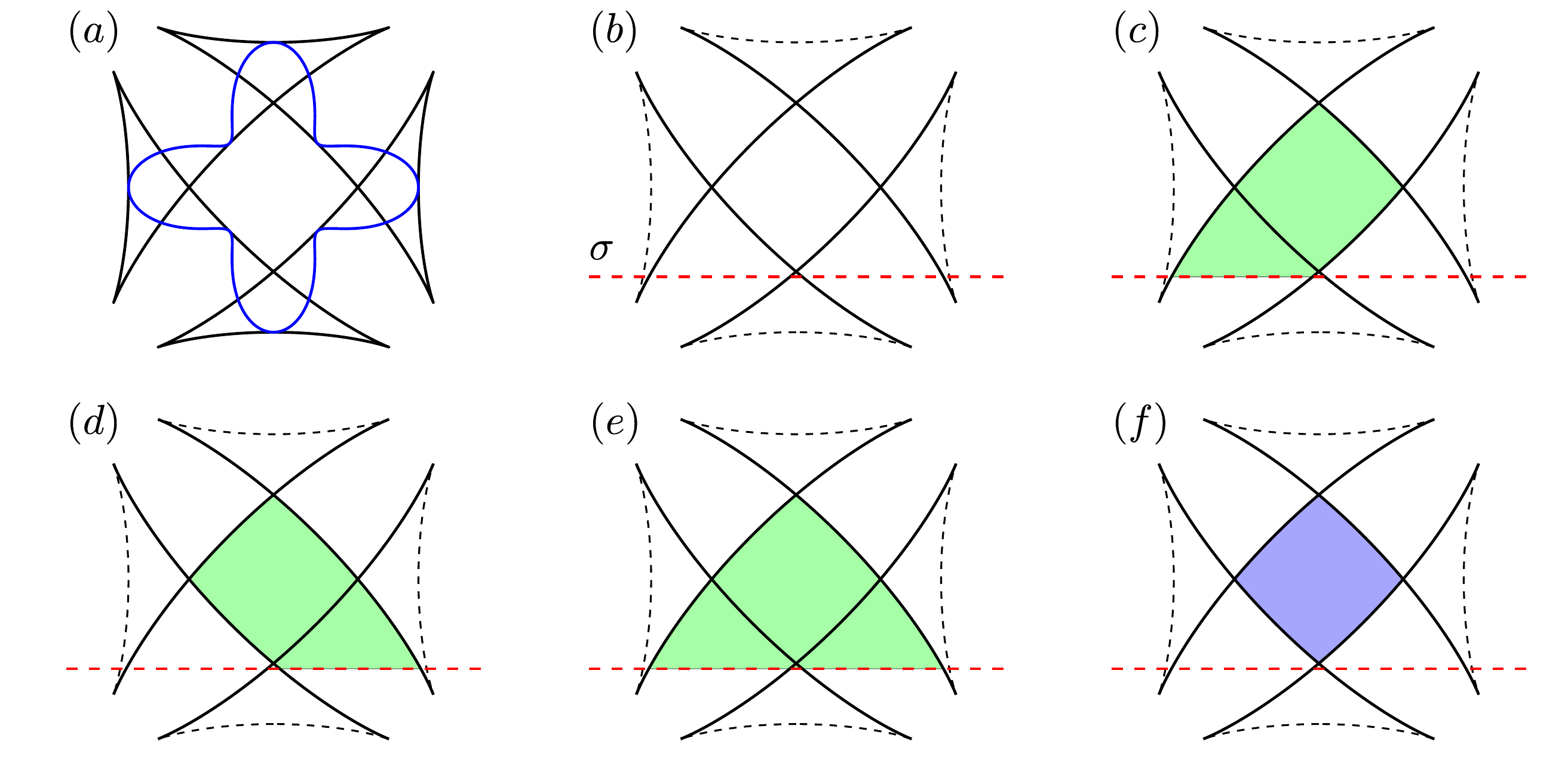}
    \caption{Generalized Winterbottom construction for surface energy $\hat{\gamma}(\theta)=1+0.4\cos 4\theta$ and $\sigma=\cos\frac{5\pi}{6}$. Green shaded and blue shaded regions are the stable equilibrium shapes.}
    \label{fig:winterbottom construction}
\end{figure}

As it has been derived in \cite{bao2017stable}, stable equilibrium shapes of two-dimensional solid-state dewetting with anisotropic surface energies can be predicted through an approach called \textit{generalized Winterbottom construction}. If the surface energy is strongly anisotropic, there may exist multiple stable equilibrium island shapes. The generalized Winterbottom construction offers a simple way to construct all possible stable equilibrium shapes via the following steps \cite{bao2017stable}: \begin{enumerate}[label=(\roman*)]
    \item \textit{Draw the Wulff envelope}: for a given anisotropy $\hat{\gamma}(\theta)$, draw its $\gamma$-plot (blue line in Fig.~\ref{fig:winterbottom construction} (a)). Obtain the Wulff envelope (black line in Fig.~\ref{fig:winterbottom construction} (a)) from the $\gamma$-plot by the following formula \cite{peng1998geometry,wang2015sharp}: \begin{equation}
        \begin{array}{ll}
           \left\{\begin{aligned}
              &x(\theta)=-\hat{\gamma}(\theta)\sin\theta-\hat{\gamma}^\prime(\theta)\cos\theta, \\
              &y(\theta)=\hat{\gamma}(\theta)\cos\theta+\hat{\gamma}^\prime(\theta)\sin\theta,
           \end{aligned}\right.  &  \theta\in 2\pi\mathbb{T}.
        \end{array}
    \end{equation}
    \item \textit{Remove all unstable orientations}: Remove all unstable orientations from the Wulff envelope for which $\hat{\gamma}(\theta)+\hat{\gamma}^{\prime\prime}(\theta)<0$ (black dash line in Fig.~\ref{fig:winterbottom construction} (b)). Only “ears” in the Wulff envelope is unstable (cf. Fig.~\ref{fig:winterbottom construction}).
    \item \textit{Truncate the Wulff envelope}: Truncate the Wulff envelope with flat substrate line $y=\sigma$ (red dash line in Fig.~\ref{fig:winterbottom construction}(b)--Fig.~\ref{fig:winterbottom construction}(f)). Then stable equilibria are regions enclosed by the Wulff envelope and the substrate line (green shaded regions in Fig.~\ref{fig:winterbottom construction}(c)--Fig.~\ref{fig:winterbottom construction}(e) and blue shaded regions in Fig.~\ref{fig:winterbottom construction} (f)).
\end{enumerate}

We applied our method to simulate equilibria of different initial thin films in solid-state dewetting with the $m$-fold surface energy $\hat{\gamma}(\theta)=1+\beta\cos m(\theta-\theta_0)$. It is strongly anisotropic when $|\beta|>\frac{1}{m^2-1}$. Results are exhibited in Fig.~\ref{fig:winterbottom construction rot} and Fig.~\ref{fig:winterbottom construction 3fold rot}. The numerical equilibrium shapes are colored in blue, the initial curves are displayed in red dash-dot line, and the corresponding Wulff envelope is shown in black dash line.

\begin{figure}[htbp]
\centering
\includegraphics[width=1.0\textwidth]{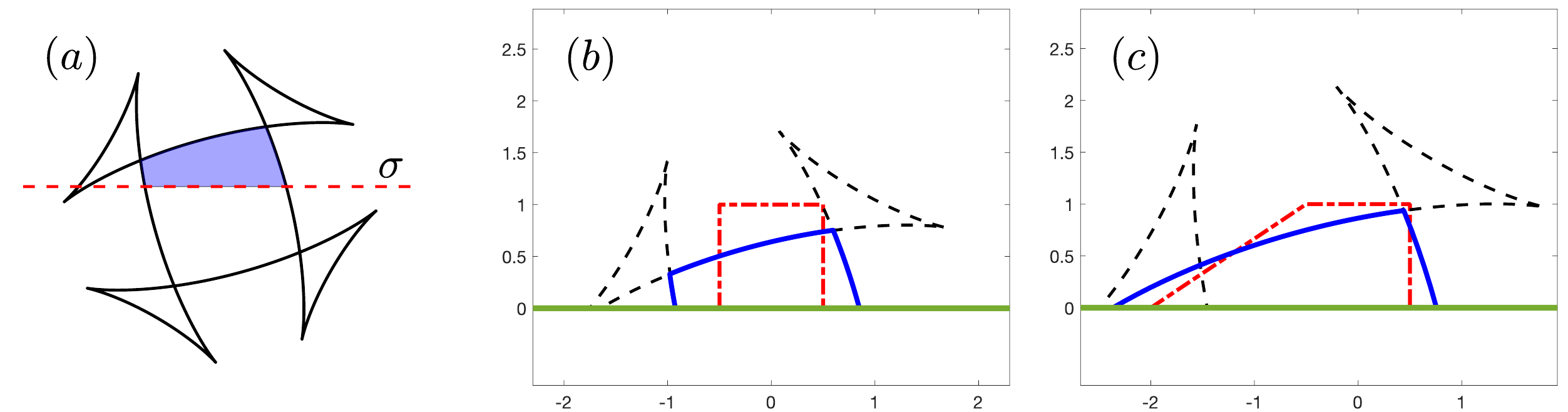}
\caption{Illustration of the generalized Winterbottom construction (a) and numerical equilibrium shapes (b)--(c) starting with different initial curves for the $4$-fold surface energy $\hat{\gamma}(\theta)=1+0.3\cos 4(\theta+\frac{\pi}{6})$ and $\sigma=0.2$.}
\label{fig:winterbottom construction rot}
\end{figure}

\begin{figure}[htbp]
\centering
\includegraphics[width=1.0\textwidth]{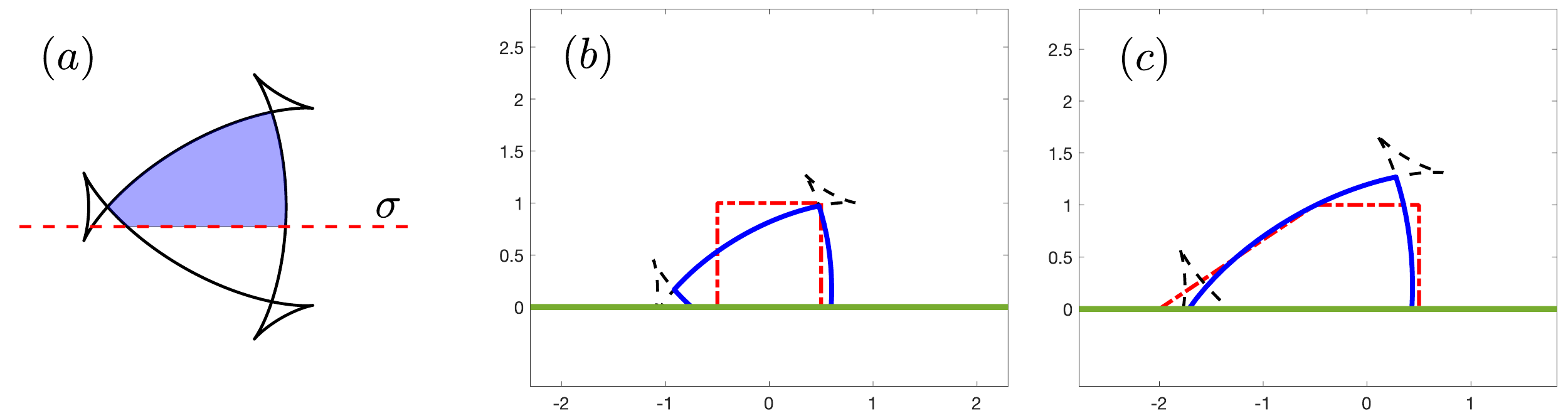}
\caption{Illustration of the generalized Winterbottom construction (a) and numerical equilibrium shapes (b)--(c) starting with different initial curves for the $3$-fold surface energy $\hat{\gamma}(\theta)=1+0.3\cos 3(\theta+\frac{\pi}{6})$ and $\sigma=-0.2$.}
\label{fig:winterbottom construction 3fold rot}
\end{figure}

Fig.~\ref{fig:winterbottom construction rot} and Fig.~\ref{fig:winterbottom construction 3fold rot} exhibit the generalized Wulff construction and the numerical equilibria by the SP-PFEM for following two asymmetric energies: \begin{itemize}
    \item Fig.~\ref{fig:winterbottom construction rot}: $\hat{\gamma}(\theta)=1+0.3\cos 4\left(\theta+\frac{\pi}{6}\right)$, $\sigma=0.2$;
    \item Fig.~\ref{fig:winterbottom construction 3fold rot}: $\hat{\gamma}(\theta)=1+0.3\cos 3\left(\theta+\frac{\pi}{6}\right)$, $\sigma=-0.2$.
\end{itemize} It can be observed from Fig.~\ref{fig:winterbottom construction rot} and Fig.~\ref{fig:winterbottom construction 3fold rot} that the numerical results and the theoretical predictions coincides very well. Unlike the symmetrized SP-PFEM in \cite{li2023symmetrized}, our method demonstrates strong performance for asymmetric surface energies. 

\section{Conclusions}\label{sec:conclusion}

We proposed a structure-preserving parametric finite element method (SP-PFEM) with optimal energy stability condition $3\hat{\gamma}(\theta)-\hat{\gamma}(\theta-\pi)\geq 0$ for anisotropic surface diffusion in two dimensions. By utilizing a symmetric surface energy matrix $\hat{\boldsymbol{Z}}_k(\theta)$, the governing equation of anisotropic surface diffusion is reformulated into a conservative weak formulation. A PFEM using piecewise linear functions for spatial discretization and an implicit-explicit Euler method for the temporal discretization is employed for this weak formulation. The SP-PFEM is second-order accurate in space, first-order accurate in time and proven to be area conservative at the fully discrete level. With a sharp estimate lemma, we established a global upper bound to explicitly characterize the existence of the minimal stabilizing function $k_0(\theta)$ under the condition $3\hat{\gamma}(\theta)-\hat{\gamma}(\theta-\pi)\geq 0$, without requiring any additional condition. Then a required local energy estimate is established, demonstrating that our method is unconditionally energy-stable with sufficiently large $k(\theta)\geq k_0(\theta)$. Since this condition is also a necessary condition for the local energy estimate, our method is considered optimal and cannot be further improved. Extensive numerical results are presented to demonstrate the efficiency, accuracy and structure-preserving properties of the proposed method. Additionally, the proposed method can be transformed into the symmetrized SP-PFEM from \cite{bao2023symmetrized2D} under the $\gamma(\boldsymbol{n})$ formulation. Our analysis indicates that the symmetry condition in \cite{bao2023symmetrized2D} can also be improved to $3\gamma(\boldsymbol{n})-\gamma(-\boldsymbol{n})\geq 0$.

\begin{acknowledgements}
The authors sincerely appreciate Professor Weizhu Bao for his valuable comments and suggestions. 

The work of Li is supported by Alexander von Humboldt Foundation.

The work of Ying is supported by Shanghai Science and Technology Innovation Action Plan (Basic Research Area) -- Project-ID 22JC1401700 and National Natural Science Foundation of China (Mathematical Sciences Division) -- Project-ID 12471342. The work is also partially supported by National Key R\&D Program of China -- Project-ID 2020YFA0712000 and the Fundamental Research Funds for the Central Universities.

The work of Zhang is supported by Chinese Scholar Council (CSC) -- State Scholarship Fund No.~202306230346 and Shanghai Jiao Tong University -- Zhiyuan Honors Program for Graduate Students -- Project-ID 021071910051. 
\end{acknowledgements}

%
%



\section*{Statements and Declarations}

\section*{Data Availability}
Data will be made available on request.

\section*{Author Contributions}
All authors contributed equally. All authors read and approved the final manuscript.

\section*{Conflict of Interest}
The authors declare that they have no conflict of interest.

\end{document}